\documentclass[12pt,reqno]{amsart}
\usepackage{amssymb,amsmath,dsfont}
\usepackage{pgfplots}
\pgfplotsset{compat=1.14}
\usepackage{enumitem}
\usepackage{calc}
\usepackage{mathrsfs}

\title{Once-reinforced random walk in high dimensions}
\date{}

\usepackage[margin=1in]{geometry}
\usepackage[all]{xy}
\usepackage{amssymb}
\usepackage{amsfonts}
\usepackage{amsthm}
\usepackage{amsmath}
\usepackage{hyperref}
\usepackage{mathtools}
\usepackage{url}
\usepackage{xcolor}
\usepackage{dsfont}
\usepackage{pgfplots}
\usepackage{bbm}

\newtheorem{thm}{Theorem}
\newtheorem*{thm*}{Theorem}

\newtheorem{lem}{Lemma}[section]  
\newtheorem{prop}[lem]{Proposition}
\newtheorem{cor}[lem]{Corollary}
\newtheorem{definition}[lem]{Definition}
\newtheorem{claim}[lem]{Claim}
\newtheorem*{h1}{Assumption (H1): Transition probabilities}
\newtheorem*{h2}{Assumption (H2): Demon relaxed times}
\newtheorem*{definition*}{Definition}

\theoremstyle{remark}

\newtheorem{remark}[lem]{Remark}

\newcommand{\PP}{\mathbb{P}}
\newcommand{\EE}{\mathbb{E}}

\newcommand{\ZZ}{\mathbb{Z}}
\newcommand{\NN}{\mathbb{N}}

\newcommand{\var}{\mathrm{Var}}
\newcommand{\eps}{\varepsilon}

\DeclareMathOperator{\capt}{Cap}
\DeclareMathOperator{\Rel}{Rel}
\DeclareMathOperator{\len}{len}
\DeclareMathOperator{\DemonRel}{Rel_{\mspace{-1mu}\demon\mspace{-3mu}}}

\newcommand{\demonint}[1][1]{
  \mathord{\smash{%
    \begin{tikzpicture}[baseline=(c.base), scale=#1]
      \node (c) at (0,0) {\strut};
      \draw (0.0,0.0) circle [radius=0.13];
      \draw (-0.13,0.05) arc [start angle = 225, end angle = 160, radius = 3pt];
      \draw (-0.08,0.11) arc [start angle = 250, end angle = 230, radius = 0.3];
      \draw (0.13,0.05)  arc [start angle = 315, end angle = 380, radius = 3pt];
      \draw (0.08,0.11)  arc [start angle = 290, end angle = 310, radius = 0.3];
    \end{tikzpicture}%
  }}%
}

\newcommand\demon{{\mathbin{%
  \mathchoice{\demonint}{\demonint}%
    {\demonint[0.65]}%
    {\demonint[0.4]}}}}

\newcommand{\W}{\mathchoice{W_{\mspace{-4mu}\demonint[0.65]\mspace{-3mu}}}
  {W_{\mspace{-4mu}\demonint[0.65]\mspace{-3mu}}}
  {W_{\mspace{-4mu}\raisebox{-1.5pt}{$\demonint[0.4]$}\mspace{-3mu}}}
  {W_{\mspace{-4mu}\demonint[0.4]\mspace{-3mu}}}}

\newcommand{\restrict}[2]{\left.{#1}\vphantom{\big|}\right|_{#2}}

\newcommand{\spend}{\tau_{\rm spend}}
\newcommand{\inner}{\tau_{\rm spend}(R)}
\newcommand{\ED}{\Omega_{\demon}}

\definecolor{myblue}{rgb}{0.09,0.32,0.44} 
\hypersetup{pdfborder={0 0 0},pdfborderstyle={},colorlinks=true,linkcolor=myblue,citecolor=myblue,urlcolor=blue}

\numberwithin{equation}{section}

\author{Dor Elboim}

\author{Gady Kozma}

\begin{document}

\begin{abstract}
    We study the once-reinforced random walk on $\mathbb Z^d$, which is a self-interacting walk that has a higher probability to cross edges that were already visited. We prove that the walk is transient when $d\ge 6$ and when the reinforcement is small, establishing a conjecture of Sidoravicius in these dimensions. Moreover, in this case we prove that the walk behaves diffusively and can be coupled with Brownian motion.

    One of the main ideas in the proof is a certain capacity estimate which shows that the trajectory of the walk is nowhere heavy. We also use a game-theoretic-type ingredient that we call ``the demon" to force spatial independence in the process.
\end{abstract}

\maketitle

\section{Introduction}

Once-reinforced random walk (ORRW) is perhaps the simplest self-interacting random walk one can imagine. The walker only remembers edges it visited before and has some preference for them. Edges never visited are given weight 1, while visited edges are given weight $1+a$ for some $a>0$ (see \S\ref{sec:prelim} for details). Defined already in 1990 \cite{davis1990reinforced}, it has proved to be much harder to analyze than anticipated. Further, it came as a great surprise when  Sidoravicius conjectured (circa 2010) that on $\mathbb Z^d$ with $d\ge 3$ it undergoes a \emph{phase transition} in the parameter $a$: When $a$ is small the walk is transient, while when $a$ is large the walk is recurrent (we will describe the conjecture in more details below). 

In this paper we prove the first part of this conjecture when $d\ge 6$. Moreover, we prove that the walk behaves diffusively and converges to a Brownian motion.

\begin{thm}\label{thm:1}
    Suppose that $d\ge 6$ and let $a>0$ sufficiently small. Then once-reinforced walk $W$ on $\mathbb Z ^d$ with reinforcement $a$ is transient. Moreover, there is a constant $\sigma \in [\frac{0.9}{\sqrt{d}},\frac{1.1}{\sqrt{d}}]$ depending on $a$ and $d$ such that 
    \begin{equation*}
        \Big\{ \frac{W(st)}{\sqrt t}\Big\}_{s\ge 0} \to \{\sigma B(s)\}_{s\ge 0}, \quad\textrm{as } t\to \infty, 
    \end{equation*}
    where $\{B(s)\}_{s\ge 0}$ is a standard Brownian motion in $\mathbb R^d$. Here the convergence is weak convergence in $C[0,T]$ for all $T>0$.
\end{thm}

We believe that the approach in this paper provides a more robust way to analyze self-interacting walks in high dimension, and that it applies to other models. But for concreteness we focus on ORRW.

\subsection{Sidoravicius' conjecture}
To understand the conjecture better, let us first consider the case that the time $t$ is fixed and $a$ is allowed to vary. It is straightforward to check that if $a\ll 1/t$ then the reinforcement does not affect the process at all and the process behaves like simple random walk. If $a\gg t$ then the opposite occurs, and the walk gets stuck on the first edge for the entire time of the process. 

More interesting behaviour occurs if $t^{0.99}\ll a\ll t$. In this case the walk visits new edges quite rarely, and spends the time between any two such events walking on the set of already visited edges. Further, between any two `new edge' events it has time to completely mix on the set of visited edges. Hence it is not difficult to see that the set of visited edges behave like an aggregate that grows by the following process. At every step choose an edge with at least one vertex in the aggregate uniformly and add it to the aggregate. This is very similar to the growth of a metric ball in \emph{first passage percolation} with exponential edge weights (only the rate of addition of edges with both end points in the aggregate differs, the behaviour of the vertex set is identical). This is also called the Eden model \cite{eden1961}.

Thus Sidoravicius' conjecture states that a similar behaviour happens when $a$ is fixed and $t$ tends to infinity, in $d\ge 3$. When $a$ is fixed and small, the walk should look like simple random walk (and, of course, this is the result of this paper, for $d\ge 6$). When $a$ is fixed and large the set of visited edges should look like a growing aggregate (not necessarily with the same shape as the first passage percolation ball). A heuristic renormalisation argument that we learned from Vincent Beffara says that at time $t$ the radius of the ball should be approximately $t^{1/(d+1)}$ and that a typical point in it should have been visited approximately $t^{1/(d+1)}$ times. Beffara's habilitation \cite{beffara2011} is also the first place the conjecture appeared in print. See \cite{kozma2012reinforced,dembo2021averaging} for some simulation results. 

Several simplified models have been suggested in order to study the `growing aggregate' phase of once-reinforced random walk. One is excited to the center, in which, whenever the walker encounters a new vertex it has a drift in the direction of its starting point. In this model the $t^{1/(d+1)}$ law can be understood via a (still heuristic) comparison of drifts argument. Another growth model defined by the exits of random walk is \emph{internal DLA}, see \cite{lawler1992idla,asselah2013idla,sheffield2014idla} see also the related \cite{lucas2020idla}. For internal DLA (and the model of \cite{lucas2020idla}) it is easy to see that the radius of the ball is $t^{1/(d+1)}$, where $t$ is the total time spent by all the walkers used to construct the aggregate.

It is interesting to note that the most well-understood reinforcement scheme, \emph{linear reinforcement}, also has a phase transition, but the `compact' phase looks quite different. In this case the process is typically within \emph{constant} distance from its starting point. This kind of `time independent' behaviour is, of course, impossible for once-reinforced random walk. See \cite{tarres2015lrrw,angel2014lrrw,xiaolin2017lrrw} for the compact phase and \cite{spencer2010susy,disertori2015lrrw} for the random walk-like phase of this model.

\subsection{Previous work}

The one-dimensional case of once-reinforced random walk is very similar to the so-called \emph{excited random walk}. For this model see the original paper \cite{davis1990reinforced} as well as \cite{benjamini2003excited,basdevant2008excited}. Specifically about once-reinforced, see Pfaffelhuber and Stiefel \cite{pfaffelhuber2021range} who showed that the walk typically reaches distance $\sqrt{t}$, but does not satisfy the central limit theorem for any reinforcement. They also determined the scaling limit of its range.

A more difficult case (but which is still one-dimensional in character) is the ladder graph $\mathbb Z \times \{1,\dots ,d\}$. On this graph Sellke \cite{sellke2006recurrence} proved that the walk is recurrent when the reinforcement is small and Vervoort \cite{vervoort2000games} and Kious, Schapira and Singh \cite{kious2021once} proved that the walk is recurrent when the reinforcement is large. The problem is still open for intermediate reinforcement.

Durrett, Kesten and Limic \cite{durrett2002once} proved that the walk is transient on regular trees for any reinforcement. Collevecchio \cite{collevecchio2006transience} proved that the walk is transient on a supercritical Galton-Watson tree. Kious and Sidoravicius \cite{kious2018phase} studied the once-reinforced walk on certain trees of  polynomial growth. They showed that the walk exhibits a phase transition between recurrence and transience similar to the conjecture above on $\mathbb Z ^d$. Later, Collevecchio, Kious and Sidoravicius \cite{collevecchio2020branching} found an exact formula for the critical reinforcement for general trees, using a geometric quantity that they named the branching-ruin number.

Very recently, Collevecchio and Tarr\`es \cite{collevecchio2025once} studied the once-reinforced walk on $\mathbb Z^d$ and proved that the volume of its range at time $t$ is typically larger than $t^{d/(d+2)}$ for any reinforcement. They also proved that the walk is transient on non-amenable graphs when the reinforcement is small.

Dembo, Groisman, Huang and Sidoravicius \cite{dembo2021averaging} studied a closely related model that mimics the behavior of ORRW on $\mathbb Z ^d$ when the reinforcement is large. We refer the reader to this paper for additional discussion, open problems and simulations of ORRW.

\subsection{The method of proof}\label{sec:sketch}
The basic scheme for the proof of the theorem was introduced in the work of the first author and Sly on the interchange
process \cite{elboim2022infinite} and was later used in \cite{elboim2025diffusivity}
to analyze the mirror model. It is a multiscale induction, namely, one makes some assumptions on the random walk up to time $T$, proves that the same assumptions hold for some larger time, and concludes inductively that they hold for all time. Very roughly, the induction assumptions are that our process `looks like random walk'. We wish to use the fact that in high enough dimension, random walk, at a typical time, with positive probability does not intersect its past. We call such times \emph{relaxed}. Namely, $t$ is relaxed if the geometry of $W[0,t]$ is such that the probability that $W[0,t]\cap W[t+1,\infty)=\emptyset$, conditioned on $W[0,t]$, is bigger than some positive constant. Roughly speaking, the inductive assumptions say that once-reinforced random walk has many relaxed times. The exact inductive assumptions can be seen in \S~\ref{sec:assump}, but let us continue this heuristic discussion without all details, for now. 

The first steps of the proof make some conclusions from the existence of many relaxed times. In \S~\ref{sec:concat} we show that the inductive assumption allows to couple a process of length $2t$ to the \emph{concatenation} of two independent processes of length $t$ with small error. Then in \S~\ref{sec:coupling} we use the concatenation property to couple our process $W$ to a concatenation of many independent much shorter processes, and conclude from that that $W$ satisfies a \emph{central limit theorem}. This part uses \emph{Koml\'os-Major-Tusn\'ady (KMT) coupling}. 

Of course, to make the induction work we need to prove that there are many relaxed times for a longer walk. Let us now make the definition of relaxed time more explicit (but still not with all details). We say that $t$ is relaxed if, when one looks at a cube $\Lambda$ of side length $r$ centered at the current position of the walker, $W(t)$, we have that $\Lambda \cap W[0,t]$ is quite small --- for \emph{simple random walk} we expect an analog quantity to be $r^2$, and we wish to show something close for once-reinforced random walk.

The proof uses two completely different arguments, one to handle relatively large $r$ and one to handle small $r$. For large $r$, we use a capacity argument.
Assume, for the sake of simplicity, that the capacity is defined by
\begin{equation}\label{eq:into-cap}
\capt(A)=\sum_{x\in A}\mathbb{P}^{x}(W\cap A=\emptyset)
\end{equation}
i.e.\ the sum over all $x\in A$ of the probability that once-reinforced
random walk started from $x$ never hits the set $A$ (of course,
except at time 0). The actual definition of capacity is a little more
complicated (see \S~\ref{sec:heavy}). The argument then claims that every time the capacity of $W$ increases, 
there is also a probability at least proportional to $\capt(W[0,t+1])-\capt(W[0,t])$ to escape the past. This means that the
event of increasing the capacity cannot happen too often before the
process detaches from its past. We perform this calculation \emph{locally}
and get that in a box $\Lambda$ of size $r$ we have ${\rm Cap}(W\cap\Lambda)\le r^{2}\log^{6}r$
with high probability, even for $W$ much longer than $r$. See Lemma~\ref{lem:small-capacity} for the details.

Once the capacity has been estimated we translate this to a volume
estimate. It is well-known for the usual capacity (defined as in
(\ref{eq:into-cap}) but for simple random walk rather than for once-reinforced walk) that
\[
c|A|^{1-2/d}\le\capt(A)\le|A|
\]
for every $A$. We repeat the classic proof (of the first inequality) for once-reinforced walk and get a local volume estimate. The reason our proof needs 6 dimensions is due to the 
inefficiency of passing from capacity to volume. In 5 dimensions, from an estimate of the form $\capt(A)\lesssim r^2$ we can only get a volume estimate of $|A|\lesssim r^{10/3}$, and sets of this size might be hittable by a second random walk. To overcome this we would have needed to rewrite the induction scheme using hitting probabilities rather than volume, which would have been less geometrically intuitive.

The argument used for smaller $r$ requires a stronger
inductive assumption. We make the assumption that once-reinforced
random walk has good behaviour inside a given box
$\Lambda$ even when its entry points into $\Lambda$ are decided by an
\emph{adversary} (we nickname our adversary `the demon' and denote
it by $\demon$, and we denote the process described by this adversary
by $\W$). The reason for this is to make different blocks \emph{effectively independent} in the following sense. The demon's optimal strategy, when it is confronted with several blocks, is to use the optimal strategy in each block, making them independent. Since once-reinforced random walk \emph{outside} the blocks is also a demon, its behaviour is bounded by that of the optimal demon. We get that, with high probability,
$\W$ has many relaxed times in almost all boxes $\Lambda$. See \S~\ref{sec:demons} for details. (The actual proof does not mention `optimal strategies', instead a supermartingale argument is used, see \S~\ref{sec:super})

Finally, for the remaining boxes (which are few and very small) we
use a \emph{volume exhaustion argument}. Since random walk in a box
$\Lambda$ of size $r$ has positive probability to exit the box after
$r^{2}$ steps (regardless of its starting position), the environment,
no matter how complicated, cannot `force' the walk to stay in the
box without `wasting' at least $r$ edges. After $r^{d-1}$ such attempts,
it runs out of steam. See \cite{amir2008excited} for a similar use of a volume
exhaustion argument (not as part of a multiscale induction scheme). 

Let us stress once more that the demon argument is only needed for smaller $r$. Even though the induction assumptions in \S~\ref{sec:assump} are formulated for the demon process $\W$, the first step in \S~\ref{sec:concat} is to drop the demon and show that the induction assumption holds for the process $W$ in `virgin environment'. The proof then proceeds through \S\S~\ref{sec:concat}-\ref{sec:coupling} studying the walk $W$, returning to $\W$ in \S\S~\ref{sec:heavy}-\ref{sec:demons}.

Thus, the properties of once-reinforced random walk we used
are
\begin{enumerate}
\item That it is invariant to translations and lattice rotations;
\item That it is a perturbation of random walk for short time;
\item The fact that if it does not intersect its past, it is independent
of it;
\item And the fact that it does not get stuck in a box much more than its volume,
regardless of its starting configuration.
\end{enumerate}

\subsection{Comparison with \texorpdfstring{\cite{elboim2022infinite}}{Elboim and Sly, 2022}}
The proofs in this paper and in \cite{elboim2022infinite} share the main structure but have several important differences. An important common element is the definition of a \emph{relaxed time}, explained in the previous section. Further, the exact requirement is that \emph{90\% of the times are relaxed with super-polynomially high probability}. See (\ref{eq:spend}) here and \cite[\S 3]{elboim2022infinite}. In both papers, this requirement is used to show that a process that has run for time $t+s$ can be coupled to the \emph{concatenation} of two independent processes running for time $t$ and $s$. From the concatenation property we deduce the central limit theorem as explained in the previous section.

The key difference between the two proofs is the replacement of the \emph{escape algorithm}, which is the main idea of \cite{elboim2022infinite}, with the capacity argument explained in the previous section. The escape algorithm (see \cite[\S 5]{elboim2022infinite}) is used to show that the process never gets `stuck', but rather always has a reasonable probability to get away from its past and continue independently. It uses inherently the fact the the interchange process studied in \cite{elboim2022infinite} is a \emph{loop system}, and hence it cannot be used here. It was replaced wholesale with the capacity argument, which fills a similar role.

Another place where the fact that our process is not a loop system made a difference is in the `independence between blocks' part of the argument. In \cite{elboim2022infinite}, when you consider the interchange process as a loop system, different blocks are \emph{truly independent}. Here no such independence is available. As explained in the previous section, we introduced the demon to make different blocks effectively independent. 

To implement these changes the inductive assumptions and the entire proof in \cite{elboim2022infinite} have to be completely reconstructed.

Finally, let us note a relatively technical point related to the volume exhaustion argument. Here the volume exhaustion says `once you have crossed all edges in a box, any further entry into it is identical to simple random walk'. This fact holds for any value of $a$, our perturbation parameter, uniformly in $a$. In \cite{elboim2022infinite} the perturbation parameter (denoted in \cite{elboim2022infinite} by $\beta$) affects the volume exhaustion. Increasing this parameter makes the volume exhaustion argument worse off, leading to some technical difficulties in initiating the induction. Here the induction base is trivial. We believe that the techniques in this paper can give an alternative proof for the interchange process, but did not want to encumber this paper with the technicalities necessary for the induction base.

\subsection{Acknowledgments}
We thank Hugo Duminil-Copin and Omer Angel for helping us with an earlier version of Lemma~\ref{lem:Cap-Vol}. We thank Ron Peretz and Eilon Solan for their game-theoretic perspective on the demons and on Lemma~\ref{lem:product}. Finally, we thank Amir Dembo, Allan Sly, Sourav Chatterjee, Amol Aggarwal and Dan Mikulincer for fruitful discussions about once-reinforced random walk and strong versions of KMT.

\subsection{Preliminaries}\label{sec:prelim}
We start with a formal definition of ORRW. Let $G$ be a graph and let $a>0$. For a subset of edges $\gamma$ (think of them as `reinforced edges'), a vertex $u$ and a neighbour $v\sim u$, define the transition probability from $u$ to $v$ by
\begin{equation}\label{eq:a}
P_{\gamma,a}(u,v)=
\frac{1+a\mathds{1}\{\{u,v\}\in\gamma\}}{\sum_{w\sim u}\big( 1+a\mathds{1}\{\{u,w\}\in\gamma\}\big) } .
\end{equation}
ORRW is then defined as the process $W:\mathbb N\to V$ with $\PP(W(t+1)=v\mid W(0),\dots ,W(t))=P_{W[0,t],a}(W(t),v)$, where $W[0,t]$ is the set of edges crossed by $W$.

It will sometimes be convenient to write the step as a function of a number $U\in[0,1)$. In this case we order the neighbours of $u$ with some fixed order, $v_1 ,\dotsc,v_{\deg u}$, let $p_1,\dotsc,p_{\deg u}$ be the probabilities for the $\deg u$ neighbours and take 
\begin{equation}\label{eq:with U}
W(t+1)=v_i\iff U\in\big[\textstyle{\sum_{j< i}p_j,\sum_{j\le i}p_j}\big).
\end{equation}
If $U$ is uniform on $[0,1)$ then this gives the same probabilities.

For a box $\Lambda \subseteq \mathbb Z ^d$  we denote by $\partial \Lambda $ the set of edges with one vertex in $\Lambda $ and one vertex in $\mathbb Z ^d\setminus \Lambda $. We denote by $\overline{\Lambda }$ the union of $\Lambda $ with all vertices with a neighbor in $\Lambda $.

A path $\gamma$ in $\mathbb Z^d$ is a function $\gamma:\{0,\dotsc,n\}\to \mathbb Z^d$ such that $\gamma(i)\sim\gamma(i+1)$ for all $i$, where $u\sim v$ means that $u$ and $v$ are nearest neighbors in $\mathbb Z^d$. We denote $\len\gamma\coloneqq n$. For a box $\Lambda \subseteq \mathbb Z^d$ we say that a $\gamma:\{0,\dotsc,n\}\to\overline{\Lambda }$ is a teleporting path (sometimes shortened to teleporter) if it is a path in $\overline{\Lambda }$ which is allowed to have two consecutive vertices in $\overline{\Lambda }\setminus \Lambda $ even if they are not neighbours, but not allowed to have three consecutive vertices in $\overline{\Lambda }\setminus \Lambda $. Namely, when the path reaches $\overline{\Lambda }\setminus \Lambda $ it can either jump to its neighbor in $\Lambda $ or `teleport' to any point in $\overline{\Lambda }\setminus \Lambda$ from which it then enters $\Lambda $. 

For a path or a teleporter $\gamma$, we denote by $\gamma[a,b]$ the set $\{\gamma(t):\forall t\in [a,b]\}$. (Sometimes we will also denote by $\gamma[a,b]$ the set of edges crossed by $\gamma$, namely $\{(\gamma(t),\gamma(t+1)\}$ for a path and $\{ (\gamma(t),\gamma(t+1)) : \gamma(t)\in\Lambda \textrm{ or }  \gamma(t+1)\in\Lambda\}$ for a teleporter --- we hope no confusion will arise from this use).

For a box $\Lambda \subseteq \mathbb Z^d$ we define $\restrict{\gamma}{\Lambda}$ to be the teleporter one gets by restricting $\gamma$ to the box $\Lambda$ (including the entries and exits from it) and `shortening time' to ignore other steps. Formally, we define
\[
  t_{i+1}=\begin{cases}
    t_i+1 & \gamma(t_i)\in\Lambda\\
    \min \{t>t_i:\gamma(t)\in \Lambda\textrm{ or }\gamma(t+1)\in \Lambda\}&\textrm{otherwise.}
  \end{cases}
\]
Similarly, $t_0$ is defined as $0$ if $\gamma(0)\in\Lambda$, and otherwise as $\min\{t\ge 0:\gamma(t+1)\in\Lambda\}$. Finally, let $\restrict{\gamma}{\Lambda}(i)=\gamma(t_i)$. It is easy to check that if $\gamma$ is a path then $\restrict{\gamma}{\Lambda}$ is a teleporter, and similarly if $\gamma$ is a teleporter in a box strictly larger than $\Lambda$.

When we write $\|\cdot\|$ we mean the $l_2$ norm (this will usually not matter, and when it does, we will use $\|\cdot\|_2$ instead). When we write $d(v,A)$ for a vertex $v$ and a set of vertices $A$ we mean $\min\{\|v-a\|_2:a\in A\}$ and similarly for $d(A,B)$ for two sets of vertices.

We use $C$ and $c$ for constants which depend only on the dimension and independent of any other parameter (such as $a$ and $T$). Their values may change from line to line, and even within the same line. We use $C$ for constants which are large enough and $c$ for constants which are small enough. Similarly, constants implicit in statements such as `$a$ is sufficiently small' may depend on the dimension. Very often we will have expressions like $Ce^{-cT}$, where $T$ is some parameter sufficiently large --- in these cases we will omit the first $C$ and just write $e^{-cT}$, which we are allowed to because it will hold for a smaller value of $c$ and sufficiently large $T$. We use $A\approx B$ as a short for $c\le A/B\le C$.

\section{Setting up the induction}

Throughout the paper we fix the exponents
\begin{equation*}
    \kappa :=3.5, \quad \nu := 0.01, \quad \epsilon := \frac{1}{10^4d} ,\quad \delta := \frac{1}{10^5d^2}.
\end{equation*}

See Remark~\ref{rem:exponents} below regarding the choice of these exponents and their meaning. The constants $C,c>0$ might depend on $\kappa ,\nu ,\epsilon ,\delta $ which are explicit functions of $d$.

\subsection{Demons}\label{sec:demon}
Let $\Lambda \subseteq \mathbb Z ^d$ be a box. Let us start with an intuitive description of a demon and later give the precise definition. A demon living outside $\Lambda $ repeatedly chooses a starting edge $e\in \partial \Lambda $ from which a once-reinforced walk enters $\Lambda $ and evolves in the environment created by previous excursions until it exits $\Lambda $. The demon can see the past excursions when making its decision of the next starting edge. Only for the first step it may choose either an edge in $\partial \Lambda $ or an arbitrary point in $\Lambda$. It can use external randomisation.

We proceed with the rigorous definition. Let $\mathscr{P}(\Lambda ) $ be the set of teleporting paths (of any finite length) in $\Lambda$. Let $\Omega$ be a probability space (this space will be referred to as the `external randomisation'). A demon $\demon$ is a function $\demon:\mathscr{P}(\Lambda)\times\Omega\to\Lambda\cup\partial \Lambda$, i.e.\ the output of the demon can be either a vertex (in $\Lambda$) or an edge (in $\partial\Lambda$). There will be a restriction on the allowed values in the case that $\demon(\gamma)\in \Lambda$ which will be detailed below. 

Any demon $\demon$ defines a \emph{demon random walk} $\W$ (which is a random teleporter in $\Lambda$) as follows. If $\demon(\emptyset)\in\Lambda$ let $\W(0)=\demon(\emptyset)$ (here $\emptyset$ is the empty path). Otherwise, $\demon(\emptyset)=\{x,y\}$ with $x\notin \Lambda $ and $y\in \Lambda $ and we set $\W(0)=x$ and $\W(1)=y$. 

To continue the definition of $\W$, if $\W(t)\in \Lambda$ we let $\W(t+1)$ be a random neighbour of $\W(t)$ chosen with $\PP(\W(t+1)=x \mid \W(0),\dots ,\W(t))=P_{\W[0,t],a}(\W(t),x)$ where $P$ is from \eqref{eq:a}. The random choice is independent of all previous choices, as well as from the randomness in the probability space $\Omega$ of the demon. 
If $\W(t)\notin\Lambda$ we apply the demon. If $\demon(\W[0,t])\eqqcolon x\in \Lambda$ then it must be the case that $x\sim \W(t)$ (this is the restriction alluded to above) and we set $\W(t+1)=x$. Otherwise $\demon(\W[0,t])\eqqcolon\{x,y\}$ with $x\notin \Lambda $ and $y\in \Lambda $ and we set $\W(t+1)=x$ and $\W(t+2)=y$.

In the literature, such a process (i.e.\ a single player playing vs.\ randomness) is called a Markov Decision Process (MDP). We refer the reader to the book by Puterman~\cite{puterman1990markov} on MDPs and in particular to the definition and basic properties in \cite[\S~2.1]{puterman1990markov}.

We will sometimes refer to $W$ as `ORRW in virgin environment' when we want to stress the difference from $\W$.

\subsection{Heavy blocks and relaxed times}

We will use the following key definitions in the inductive assumptions and throughout the paper. 

\begin{definition}[Heavy blocks]
 A block $\Lambda =u+[-r,r]^d$ is called heavy with respect to a walk $W$ at time $t$ if 
\begin{equation*}
    \big| \{W(s):s\le t\} \cap \Lambda  \big| \ge r^{\kappa}.
\end{equation*}
\end{definition}

\begin{definition}[Relaxed times]\label{def:relaxed}
    A time $t\ge 0$ is called $R$-locally relaxed (with respect to the walk $W\!$) if for any $r\in [a^{-1/3},R]$ the block $W(t)+[-r,r]^d$ is not heavy at time $t$. (Recall that $a$ is the reinforcement parameter).

    We denote by $\Rel (R) \subseteq \mathbb N$ the set of $R$-locally relaxed times. We define $\DemonRel(R)$ in the same way as $\Rel(R)$ but for $\W$ instead of $W$.
\end{definition}

\subsection{The inductive assumptions}\label{sec:assump}
The inductive assumptions at time $T$ are:

\begin{h1}
\label{h1}
For any $v\in \mathbb Z ^d$ we have that 
\begin{equation*}
    \mathbb P (W(T)=v)\le\min \big( T^{-d/2+\nu },\|v\|^{-d+2\nu} \big).
\end{equation*}
\end{h1}

\begin{h2}
\label{h2}
Let $R=\lfloor \sqrt{T} /2 \rfloor $. The following holds for any demon $\demon$ outside of $[-4R,4R]^d$. Let $\W$ be the demon walk and define the stopping time of spending time $R^2 \log ^3R$ inside the inner box $[-R,R]^d$: 
\begin{equation}\label{eq:spend}
    \inner :=\min \big\{ t\ge 0 : |\{s\le t: \W(s)\in [-R,R]^d\}| \ge R^2 \log ^3R\big\}.
\end{equation}
We have that
\begin{equation*}
     \mathbb P \bigg( \!\! \begin{array}{cc}
        \forall t\le \inner -R^{\epsilon } \text{ such that }\W(t)\in [-R,R]^d   \\
          \text{ we have }\big|(t,t+R^{\epsilon })\cap \, {\DemonRel} (R) 
 \big| \ge  0.9R^\epsilon +1 
    \end{array}\!\! \bigg) \ge 1-e^{-\log ^2T}.
\end{equation*}
\end{h2}

\begin{remark}\label{rem:exponents}
  Let us explain the choice of exponents $\kappa ,\nu ,\epsilon ,\delta $. It is important that the heaviness exponent $\kappa$ is smaller than $d-2$, this will ensure the walk typically does not hit its history. We also need $\kappa $ to be larger than $2d/(d-2)$ which is the volume bound we obtain from the capacity estimate. This explains the restriction $d\ge 6$ in Theorem~\ref{thm:1}. The exponent $\nu >0$ is the buffer we have in Assumption~\hyperref[h1]{\bf{(H1)}} on transition probabilities. It is important to have $\nu \ll \kappa -2d/(d-2)$ in order to justify the capacity estimate and also $\nu \ll d-2-\kappa $ in order to make sure the walk does not hit its history. Next, $R^\epsilon$ is the scale of time in which we ensure relaxed times, which is roughly the size of the error of concatenation. It is useful to have $\epsilon \ll \nu $ in order to prove Assumption~\hyperref[h1]{\bf{(H1)}} using Assumption~\hyperref[h2]{\bf{(H2)}} at a shorter time (see Lemma \ref{lem:H1}).

  Finally, the parameter $\delta$ has not appeared yet explicitly, but we alluded to it in the proof sketch (\S\ref{sec:sketch}). The threshold between the scale where we use the capacity estimate and where we use the demon estimate is $R^\delta$. It is important that $\delta \ll \epsilon $ in order to prove Assumption~\hyperref[h2]{\bf{(H2)}}. See the end of the proof of Theorem \ref{thm:relaxed} on page \pageref{pg:del_eps}.
\end{remark}

\subsection{The induction base and step}

Let $T_0$ be the minimal integer after $3600^{1/\epsilon}$ such that for all $T\ge T_0$ and $v\in \mathbb Z ^d$ we have 
\begin{equation*}
    \mathbb P (X(T)=v) \le \tfrac{1}{2} \min \big( T^{-d/2+\nu }, \|v\|^{-d+2\nu } \big),
\end{equation*}
where $X$ is a simple random walk. The local central limit theorem (see, e.g., \cite[Theorem~2.11]{lawler2010random}) promises that such a $T_0$ exists and depends only on $d$. The following lemma is the induction base.

\begin{lem}\label{lem:base}
    The inductive assumptions \hyperref[h1]{\bf{(H1)}} and~\hyperref[h2]{\bf{(H2)}} hold at any time $T\in [T_0, a^{-1/(2d)} ]$ if $a$ is sufficiently small.
\end{lem}

\begin{proof}
Assumption~\hyperref[h2]{\bf{(H2)}} on relaxed times holds trivially when $T\in [T_0,a^{-1/(2d)}]$. Indeed, by Definition~\ref{def:relaxed} of relaxed times, any time $t$ is $R$-locally relaxed when $R=\lfloor \sqrt{T}/2 \rfloor \le a^{-1/3}$, since the interval $[a^{-1/3},R]$ is empty. Thus, if $T\in [T_0,a^{-1/(2d)}]$ then $\big| (t,t+R^{\epsilon })\cap {\DemonRel}(R)\big| \ge R^{\epsilon }-1 \ge 0.9R^{\epsilon }+1$, where the last inequality holds since $R= \lfloor \sqrt{T} /2\rfloor\ge \lfloor \sqrt{T_0}/2 \rfloor \ge  30^{1/\epsilon }$. It follows that the event in Assumption~\hyperref[h2]{\bf{(H2)}} holds deterministically.

We turn to prove Assumption~\hyperref[h1]{\bf{(H1)}} on transition probabilities when $T\in [T_0,a^{-1/(2d)}]$. It suffices to prove the assumption when $\|v\|\le T$ since otherwise $\mathbb P (W(T)=v)=0$. There is a coupling of $W$ with a simple random walk $X$ such that 
    \begin{equation*}
        \mathbb P \big( \forall t\le T, \ W(t)=X(t) \big) \ge 1-CaT \ge 1-\tfrac{1}{2}T^{-d},
    \end{equation*}
    where the last inequality holds since $T\le a^{-1/(2d)}$ and $a$ is sufficiently small. Thus, 
    \begin{equation*}
    \begin{split}
        \mathbb P (W(T)=v)&\le \mathbb P (X(T)=v) + \tfrac{1}{2}T^{-d} \\
        &\le \tfrac{1}{2}\min \big( T^{-d/2+\nu }, \|v\|^{-d+2\nu } \big)+ \tfrac{1}{2}T^{-d} \le  \min \big( T^{-d/2+\nu }, \|v\|^{-d+2\nu } \big),
    \end{split}
    \end{equation*}
    where the second inequality holds by the definition of $T_0$ and since $T\ge T_0$.
\end{proof}

The following theorem is the induction step.

\begin{thm}\label{thm:step}There exists some $a_0=a_0(d)$ such that for all $a<a_0$ the following holds. 
    Let $T\ge a^{-1/(2d)} $ and suppose that the inductive assumptions~\hyperref[h1]{\bf{(H1)}} and~\hyperref[h2]{\bf{(H2)}} hold at all times $T'\in [T_0,T)$. Then, \hyperref[h1]{\bf{(H1)}} and~\hyperref[h2]{\bf{(H2)}} hold at time $T$.
\end{thm}

The rest of the paper is devoted to the proof of Theorem~\ref{thm:step} and therefore from now on we fix $T\ge a^{-1/(2d)}$ and assume that \hyperref[h1]{\bf{(H1)}} and~\hyperref[h2]{\bf{(H2)}} hold at all times $T'\in [T_0,T)$, and that $a$ is sufficiently small.

Our main result (Theorem~\ref{thm:1}) is an easy consequence of Theorem~\ref{thm:step} and some of the lemmas used in its proof, so we postpone it to page \pageref{pg:proof1}.

\section{Concatenation}\label{sec:concat}

We start by showing that Assumption~\hyperref[h2]{\bf{(H2)}} on many relaxed times for a demon walk implies many relaxed times for the usual once-reinforced walk.

\begin{lem}\label{lem:abstract}
Let $W$ be once-reinforced random walk and let $\Lambda=u+[-r,r]^d$ for some $u\in\ZZ^d$ and some $r>0$. Then, there exists a demon $\demon$ outside of $\Lambda $ such that $\restrict{W}{\Lambda}\sim \W$. 
\end{lem}
(As usual, $\sim$ means `have the same distribution').
\begin{proof}This is `probabilistic abstract nonsense', but let us do it in detail nonetheless.
We use an `envelope' representation of ORRW --- this means that we imagine at every vertex an infinite stack of envelopes, and when the walker arrives at the vertex it opens the top envelope, reads the instruction where to go, and discards the envelope. Formally, for each $v\in\ZZ^d$ we define infinitely many random variables $(U_{v,n})_{n=1}^\infty$, i.i.d.\ (also independent between sites), uniform on $[0,1]$. We then define W(t) as follows. Assume at time $t$ we have that $W(t)=v$ and $|\{s\le t:W(s)=v\}|=n$. We then use $U_{v,n}$ to sample the next step according to \eqref{eq:with U}. Clearly, this is the same process as usual ORRW.

Recall that we need to define a space of external randomisation $\Omega$. We let $\Omega$ be the space realizing the variables $(U_{v,n}:n\in\NN, v\not\in\Lambda)$. If $0\in\Lambda$ we let $\demon(\emptyset)=0$. Otherwise we let $\demon(\emptyset)$ be the first edge through which $W$ enters $\Lambda$ (which is, of course, a function of $\Omega$).

For every $\omega\in\Omega$ and for every teleporter $\gamma\in\mathscr{P}(\Lambda)$ we define $\demon(\gamma,\omega)$ as follows. We let $X$ be the process that uses the variables $(U_{v,n}:n\in\NN, v\not\in\Lambda)$ to sample the next step when $X(t)\notin \Lambda $, and follows $\gamma$ inside $\Lambda$, if possible (for example, if $\gamma$ starts from a certain $\{x,y\}\in\partial\Lambda$ but $X$ first entered $\Lambda$ from some edge different from $\{x,y\}$, then it is not possible for $X$ to follow $\gamma$). If it is possible for $X$ to follow $\gamma$, i.e.\ for some $t$ we have that $\restrict{X[0,t]}{\Lambda}=\gamma$, we define $\demon(\gamma,\omega)$ as the first entry point of $X$ into $\Lambda$ after $t$. Otherwise, we give $\demon(\gamma,\omega)$ an arbitrary fictitious value (this will not affect $\W$). Similarly, if $\gamma\ne\emptyset$ and $\gamma(\len\gamma)\in\Lambda$ then we set $\demon(\gamma,\omega)$ to a fictitious value.

As for what `entry point' means, if $X$ returns to $\Lambda$ in one step, we let $\demon(\gamma)$ be the vertex in $\Lambda$ through which it returned. If $X$ takes more than one step to return to $\Lambda$, we let $\demon(\gamma)$ be the edge through which it returned. This finishes the definition of $\demon$.

It is now easy to check that, if $\W$ uses the same $U_{v,n}$ as $W$ to sample its steps in $\Lambda $, then $\W = \restrict{W}{\Lambda}$. This coupling proves the equality of distributions.
\end{proof}

\begin{lem}[Virgin relaxed times]\label{lem:virgin relaxed}
Let $W$ be ORRW starting from the origin. Then for all $t\le T$ we have that 
\begin{equation*}
    \mathbb P \big( \forall s\le t,\   \big| (s,s+t^{\epsilon })\cap \Rel(\infty) \big| \ge 0.8t^{\epsilon } \big) \ge 1-Ce^{-c\log ^2t}.
\end{equation*}
\end{lem}

\begin{proof}
  We may assume that $t$ is sufficiently large (as usual, depending on $d$) and let $r=\lfloor \sqrt{t} /4\rfloor $. Let $u\in\mathbb{Z}^d$ and let $\Lambda=u+[-r,r]^d $ and $\Lambda ^+:=u+[-4r,4r]^d$. By Lemma~\ref{lem:abstract} there is a demon $\demon $ outside of $\Lambda ^+$ such that $\W\sim \restrict{W}{\Lambda^+}$. Thus, by Assumption~\hyperref[h2]{\bf{(H2)}} at time $4r^2<T$ and with the demon $\demon$ we have 
\begin{equation}\label{eq:56}
    \mathbb P \bigg( \!\!\!\! \begin{array}{cc}
        \forall s\le r^2\log ^2r \text{ such that } W(s)\in \Lambda    \\
          \text{ we have }\big|(s,s+r^{\epsilon })\cap \, \Rel(r) 
 \big| \ge  0.9r^{\epsilon }+1 
    \end{array}\!\!\! \bigg) \ge 1-e^{-\log ^2r}.
\end{equation}
We used here that the stopping time from Assumption~\hyperref[h2]{\bf{(H2)}} is always at least~$r^2\log ^3r$.

This is already quite close to what we need but we need to get rid of the restriction $W(s)\in\Lambda$. For this we note that up to time $t$ the walk will stay in $[-t,t]^d$. Cover $[-t,t]^d$ by $Ct^d$ blocks $\Lambda = u+[-r,r]^d$. Since $W(s)$ is always contained in one of these blocks, \eqref{eq:56} and a union bound over these blocks gives
\begin{equation*}
    \mathbb P \big( \forall s\le t,\   \big| (s,s+r^{\epsilon })\cap \Rel (r)   \big| \ge 0.9r^{\epsilon }+1 \big) \ge 1-Ct^de^{-\log ^2r}.
\end{equation*}
We now want to move from time intervals of length $r^\epsilon$ to intervals of length $t^\epsilon$. We note that, if the event inside the probability holds, then we can partition an interval of length $t^\epsilon $ into intervals of length $r^{\epsilon }$ with a small remainder (here we use that $t$ and hence $r$ are sufficiently large). This gives a high density of $r$-locally relaxed times within this interval. Blocks of side length larger than $r$ are clearly not heavy by time $t\le 16r^2$ and so all these times are actually in $\Rel(\infty)$ and not just in $\Rel(r)$.
\end{proof}

\subsection{Avoiding the history}

In the following lemma we show that after a relaxed time, the walk can be coupled perfectly with an independent walk with positive probability.

We wish to generate ORRW using a sequence $U_1,U_2,\dotsc$ of i.i.d.\ uniform variables in $[0,1)$. We use \eqref{eq:with U}, but note that here it will be important that the fixed order of neighbours that we used in \eqref{eq:with U} is translation invariant.

\begin{definition}\label{def:natural}
When we talk about the \emph{natural coupling} of two walks $W$ and $W'$, this simply means that we use the same sequence of uniform variables $U_1, U_2,\dots $ to sample both walks, and the same fixed, translation invariant order on the neighbours. 
\end{definition}
For example, in the natural coupling of a once-reinforced walk $W$ and a simple random walk $X$, in each step the walks go in the same direction with probability at least $1-Ca$. The next lemma uses the natural coupling of two ORRW, one in empty environment, and one with some edges already reinforced.

\begin{lem}\label{lem:avoiding}
Let $W$ be a once-reinforced walk and let $t\ge 1$. Let $W'$ be a once-reinforced walk starting from the origin that is independent of $W[0,t]$. In the natural coupling of $\{W(t+s)\}_{s\ge 0}$ and $\{W'(s)\}_{s\ge 0}$, on the event that $t\in\Rel(\infty)$ we have
\[
\PP\big( \forall s< T, \  W(t+s)=W(t)+W'(s) \mid W[0,t] \big) > 1-Ca^{1/9}.
\]
\end{lem}

For the proof of Lemma~\ref{lem:avoiding} we will need the following lemma.

\begin{lem}\label{lem:A}
    Let $L\ge \sqrt{T_0}$ and let $A\subseteq \mathbb Z^d$ such that $|A\cap [-r,r]^d|\le r^{\kappa +\nu }$ for all $r\ge L$.  Then, we have
    \begin{equation*}
        \mathbb P \big( \exists  t\in [L^2,T), \  W(t) \in A \big) \le CL^{-1/3}.
    \end{equation*}
\end{lem}

\begin{proof}
    Let $t\in [L^2,T)$ and let $\Lambda _t:= [-\sqrt{t},\sqrt{t}]^d$. By Assumption~\hyperref[h1]{\bf{(H1)}}  we have 
\begin{equation}\label{eq:in}
        \mathbb P \big( W(t)\in A\cap \Lambda _t \big) =\sum _{x\in A\cap \Lambda _t} \mathbb P (W(t)=x) \le | A\cap \Lambda _t | \cdot t^{-d/2+\nu } \le t^{(\kappa-d+3\nu)/2 }. 
    \end{equation}
Similarly, letting $j_0:= \lfloor \log _2(\sqrt{t}) \rfloor $ we have that
    \begin{equation}\label{eq:out}
    \begin{split}
        \mathbb P \big( W(t)&\in A \setminus \Lambda _t \big)= \sum _{x\in A\setminus  \Lambda _t} \mathbb P (W(t)=x) \le   \sum _{x\in A\setminus \Lambda _t } \|x\|^{-d+2\nu } \\
        &\le \sum _{j=j_0}^{\infty }  |A\cap [-2^{j+1},2^{j+1}]^d| \cdot 2^{j(-d+2\nu )}  \le C\sum _{j=j_0}^{\infty } 2^{j(\kappa -d +3\nu )} \le Ct^{(\kappa -d +3\nu )/2 }.
    \end{split}
    \end{equation}
    Combining \eqref{eq:in} and \eqref{eq:out} we obtain that $\mathbb P (W(t)\in A)\le Ct^{(\kappa -d+3\nu )/2}\le Ct^{-1.2}$, where in the last inequality we substituted the values of $\kappa$ and $\nu $ and used that $d\ge 6$. The lemma follows from this bound and a union bound over $t\in [L^2,T)$.
\end{proof}

\begin{proof}[Proof of Lemma~\ref{lem:avoiding}]
Let $A_t:=\{W(s):s\le t\}$ be the trace of $W$ up to time $t$. In the natural coupling, $W(t+s)$ and $W'(s)$ take a step in the same direction with probability at least $1-Ca$  (we used here that the order of the neighbours is translation invariant). Thus, letting $\Omega _1:=\{\forall s\le a^{-2/3}, W(t+s)=W(t)+W'(s)\}$, we have that $\mathbb P (\Omega _1\mid  W[0,t])\ge 1-Ca^{1/3}$.

Next, let $\Omega _2:=\{\forall s\in  [a ^{-2/3},T), \ W(t)+W'(s)\notin A_t\}$. Since $t\in \Rel (\infty)$, the set $A_t-W(t)$ satisfies the volume condition in Lemma~\ref{lem:A} with $L=a^{-1/3}$. Thus, by Lemma~\ref{lem:A} $\mathbb P (\Omega _2\mid W[0,t])\ge 1-Ca^{1/9}$. This finishes the proof of the lemma since on $\Omega _1 \cap \Omega _2 $ we have that $W(t+s)=W(t)+W'(s)$ for any $s<T$. Indeed, for every $s\in [a ^{-2/3},T)$, if $W[t,t+s]=W(t)+W'[0,s]$ and $W(t)+W'(s)\notin A_t$, then the walks $W(t+s)$ and $W(t)+W'(s)$ interact with identical histories and will take the same next step in the natural coupling (again, from translation invariance). The result follows by induction on $s\in [a ^{-2/3},T)$.
\end{proof}

\subsection{Concatenation}

Let $t>0$ and $t_1,t_2\ge 0$ such that $t_1+t_2=t$. Let $W_1,W_2$ be two independent once-reinforced walks starting from the origin of lengths $t_1,t_2$, respectively. Define the concatenated walk $\tilde W$ of length $t$ as follows
\begin{equation}
    \tilde{W}(s):=\begin{cases}
        \quad \quad W_1(s), \quad &s\in [0,t_1]\\
        W_1(t_1)+W_2(s-t_1), \quad &s\in [t_1,t]
    \end{cases}.
\end{equation}

In the natural coupling of $W$ and $\tilde W$, we generate $W$ with the uniform variables $U_1,\dots ,U_t$, generate $W_1$ with $U_1,\dots U_{t_1}$, and generate $W_2$ with $U_{t_1+1},\dots ,U_t$. Hence, in this coupling we have that $W[0,t_1]=\tilde W[0,t_1]$ deterministically. The next lemma shows that $W$ and $\tilde W$ continue to be close after $t_1$ with high probability.

\begin{lem}\label{lem:concat}
    Let $t\le T$ and let $t_1,t_2\ge 0$ with $t_1+t_2= t$. Let $W_1,W_2$ be independent ORRW of lengths $t_1,t_2$ respectively. Let $\tilde W$ be their concatenation and let $W$ be ORRW of length $t$. In the natural coupling of $\tilde W$ and $W$ we have that 
\[
\mathbb{P} \Big( \max _{s\le t}\|W(s)-\tilde W (s)\|\ge t^{2\epsilon } \Big) \le Ce^{-c\log^{2}t}.
\]
\end{lem}

\begin{proof}
We may assume that $t_1\ge 1$. Let $S$ be the set of relaxed times of $W$ and let $\tilde S$ be the set of times of the form $t_1+s$ where $s$ is relaxed with respect to $W_2$. Define a sequence of stopping times by $\zeta _0=t_1$ and inductively for $j\ge 1$ 
\begin{equation*}
    \eta _{j} :=\min \big\{ s>\zeta _{j-1} : s\in S\cap \tilde S \big\},\quad \zeta _{j}:=\min \big\{ s>\eta _{j} : W(s)-W(\eta _j)\neq \tilde W(s)-\tilde W(\eta _j) \big\}.
\end{equation*}
Let $J:=\max \{j:\eta _{j+1}\le t \}$. We have that 
\begin{equation}\label{eq:sum1}
    \max _{s\le t} \|W(s)-\tilde W(s)\| \le 2(t-\zeta _{J+1})^++ 2\sum _{j=1}^{J} (\eta _{j+1}-\zeta _j).
\end{equation}
Recall the uniform variables $U_1,U_2,\dots $ used in the natural coupling and let $\mathcal F_t:=\sigma (U_1,\dots ,\linebreak[0]U_t)$. We apply Lemma~\ref{lem:avoiding} twice. First we apply it to $W$ and we get, with probability at least $1-Ca^{1/9}$, that $W(s)-W(\eta_j)=W'(s-\eta_j)$ for all $s\in[\eta_j,T]$, where $W'$ is an independent ORRW generated by $U_{\eta _{j}+1}, U_{\eta _{j}+2}\dots $. We repeat this argument for $W_2$ and get that it too is equal to $W'$ and hence $W(s)-W(\eta_j) = \tilde W(s)-\tilde W(\eta_j)$, again with probability at least $1-Ca^{1/9}$. If $a$ is sufficiently small, this last probability is at least $\frac{1}{2}$. Thus, we have $\mathbb P ( \zeta _{j} > t \mid \mathcal F _{\eta _{j}} ) \ge 1/2$
and therefore by repeated trials we have $\mathbb P ( J\ge \log ^2t ) \le 2^{-\log ^2t}$.

Next, observe that if the event of Lemma~\ref{lem:virgin relaxed} holds both for $\{W(s)\}_{s\le t}$ and $\{W_2(s)\}_{s\le t}$ (with $W_2$ extended to be of length $t$) then we have that $\eta _{j+1}-\zeta _j \le t^{\epsilon }$ for any $j\le J$, and also $(t-\zeta _{J+1})^+ \le t^\epsilon$. Thus, on these events and the event that $J\le \log ^2t$, the sum on the right hand side of \eqref{eq:sum1} is bounded by $t^{2\epsilon }$.
\end{proof}

The next result follows from Lemma \ref{lem:concat} by induction.

\begin{cor}\label{cor:concat} Let $t\le T$, $k\ge 1$ and $t_1,\dots ,t_k\ge 0$ with $t_1+\cdots +t_k=t$. Let $W_1,\dots ,W_k$ be independent ORRW of lengths $t_1,\dots ,t_k$ respectively. Let $\tilde W$ be their concatenation and let $W$ be ORRW of length $t$. In the natural coupling of $\tilde W$ and $W$ we have that 
\[
\mathbb{P} \Big( \max _{s\le t}\|W(s)-\tilde W (s)\|\ge kt^{2\epsilon } \Big) \le Cke^{-c\log^{2}t}.
\]
\end{cor}

When concatenating a short piece to a much longer one, we can ensure a concatenation error that is better in the beginning (at the expense of a worse probability bound).

\begin{cor}\label{cor:3.2}
   Let $t\le T$ and let $t_1,t_2\ge 0$ with $t_1+t_2= t$. Let $W_1,W_2$ be independent ORRW of lengths $t_1,t_2$ respectively. Let $\tilde W$ be their concatenation and let $W$ be a ORRW of length $t$. In the natural coupling of $\tilde W$ and $W$ we have that 
\[
\mathbb{P} \Big( \exists s\le t \text{ with } \|W(s)-\tilde W (s)\|\ge (t_1+s)^{3\epsilon } \Big) \le Ce^{-c\log^{2}t_1}.
\]
\end{cor}

\begin{proof}
  For every $k\ge 1$ we use Lemma~\ref{lem:concat} with $t_{\textrm{Lemma~\ref{lem:concat}}}=u=\min(2^kt_1,t)$ (and the $t_2$ of Lemma~\ref{lem:concat} being $u-t_1$). We get that, in the natural coupling, 
    \begin{equation*}
        \mathbb P \Big(\max _{s\le u}\|W(s)-\tilde W(s)\|\ge (2^kt_1)^{2\epsilon } \Big) \le Ce^{-c\log ^2u}.
    \end{equation*}
    The result follows from a union bound over $k\le \lceil \log _2(t/t_1) \rceil $ (we replaced $2\epsilon$ with $3\epsilon$ to compensate for a factor of 2, which we can if $t_1$ is sufficiently large, a fact we may assume).
\end{proof}

\section{Diffusivity}\label{sec:coupling}

In this section we use the concatenation result in order to prove diffusive estimates for the once-reinforced walk. In particular, we establish Assumption~\hyperref[h1]{\bf{(H1)}} on transition probabilities at time $T$, proving the first half of Theorem~\ref{thm:step}.

\subsection{The variance of the walk}\label{sec:var}

For any $t\ge 0$ let $v_t:={\rm Var}(W(t)_1)=\mathbb E [W(t)_1^2]$ where $W(t)_1$ is the first coordinate of $W(t)$ and where $\mathbb E [W(t)_1]=0$ by symmetry. Let 
\begin{equation*}
    \sigma :=\sqrt{v_T/T}.
\end{equation*}
\begin{prop}\label{prop:var}
    For any $t\le T$ we have  $d\cdot v_t/t\in [0.9,1.1]$. Further, $|v_t-\sigma ^2t|\le Ct^{1/2+5\epsilon }$.
\end{prop}

The rest of \S\ref{sec:var} is devoted to the proof of Proposition~\ref{prop:var}. This is done using an additional induction over time. We will assume inductively that for some $T'\le T$ we have that
\begin{equation}\label{eq:induction}
    d\cdot v_t/t\in [0.9,1.1],\quad \forall t\in [1,T')
\end{equation}
and show that it holds also for $t=T'$. For the induction step, we start with the following lemma which shows that $v_t$ is approximately additive.

\begin{lem}\label{lem:additive}
    For any $1\le s\le t\le T'$ with $s+t\le T'$ we have that 
    \begin{equation*}
        \big|v_{t+s}-v_{t}-v_{s} \big| \le Ct^{1/2+2\epsilon }.
    \end{equation*}
\end{lem}

\begin{proof}
   Let $W_1,W_2$ be two ORRW of lengths $s,t$ respectively and let $\tilde W$ be their concatenation. Note that $\mathbb E \big[\tilde W(s+t)_1^2 \big] =v_{s}+v_{t}$. Let $\Omega :=\{ \|W(t+s)-\tilde W(t+s)\|\le 2t^{2\epsilon }\}$ be the event that the concatenation works. By Lemma~\ref{lem:concat} we have that $\mathbb P (\Omega ^c) \le Ce^{-c\log ^2t}$. Thus,
    \begin{equation*}
    \begin{split}
        v_{s+t}&=\mathbb E \big[ W(s+t)_1^2 \big] \le (s+t)^2 \cdot  \mathbb P (\Omega ^c) +\mathbb E \big[ (|\tilde W(s+t)_1|+2t^{2\epsilon } )^2 \big] \\
        &\le Ce^{-c\log ^2t} + \mathbb E \big[\tilde W(s+t)_1^2 \big] +4t^{2\epsilon }\mathbb E |\tilde W(s+t)_1|+4t^{4\epsilon }\\
        &\le Ce^{-c\log ^2t} + v_{s}+v_{t} +4t^{2\epsilon }\sqrt{v_{s}+v_{t}}+4t^{4\epsilon } \le v_{s}+v_{t} +Ct^{1/2+2\epsilon },
    \end{split}
    \end{equation*} 
    where in the second inequality we used Cauchy-Schwarz and in the last inequality we used the inductive assumption \eqref{eq:induction} for $s$ and $t$. 
    Similarly, we have that 
    \begin{equation*}
        v_{s+t} \ge \mathbb E \big[ (|\tilde W(s+t)_1|-2t^{2\epsilon } )^2 \big] -Ce^{-c\log ^2t}\ge v_{s}+v_t -Ct^{1/2+2\epsilon },
    \end{equation*}
    as needed.
\end{proof}

\begin{cor}\label{cor:t/2}
    For any $t\le T'$ and $s\in [t/2,t]$ we have that 
    \begin{equation*}
        \Big| \frac{v_s}{s}-\frac{v_t}{t} \Big| \le Ct^{-1/2+2\epsilon }
    \end{equation*}
\end{cor}

\begin{proof}
    Define the function $f(s):=v_s-sv_t/t$ for $s\in [0,t]$ and let $C_0$ be the constant from Lemma~\ref{lem:additive}. We claim that for all $s\le t/2$ we have that 
\begin{equation}\label{eq:induction2}
        |f(s)|\le 2C_0t^{1/2+2\epsilon }.
    \end{equation}
We prove this using induction on $s\le t/2$ (so here we have three nested inductions). Suppose that \eqref{eq:induction2} holds for any $r<s$ and we prove it for $s$. To this end, write $t=qs+r$ with $q\ge 2$ and a remainder $r<s$. We have that 
    \begin{equation*}
        |qf(s)+f(r)| = |qv_s+v_r-v_t| \le qC_0t^{1/2+2\epsilon },
    \end{equation*}
    where in the last inequality we used Lemma~\ref{lem:additive} repeatedly $q$ times. Using \eqref{eq:induction2} for $r$ we obtain that 
    \begin{equation*}
     |f(s)|\le |f(r)|/q +C_0t^{1/2+2\epsilon } \le |f(r)|/2 +C_0t^{1/2+2\epsilon } \le 2C_0t^{1/2+2\epsilon },  
    \end{equation*}
    completing this induction. For $s\ge t/2$ we have that 
    $$|f(s)+f(t-s)|=|v_t-v_s-v_{t-s}|\le C_0t^{1/2+2\epsilon }$$
    and therefore $|f(s)|\le |f(t-s)|+C_0t^{1/2+2\epsilon } \le 3C_0t^{1/2+2\epsilon }$. This finishes the proof.
\end{proof}

\begin{cor}\label{cor:3}
    For any $s\le t\le T'$ we have that 
    \begin{equation}\label{eq:cor3}
        \Big| \frac{v_s}{s}-\frac{v_t}{t} \Big| \le Cs^{-1/2+2\epsilon }
    \end{equation}
\end{cor}

\begin{proof}
    Let $s_0:=s$ and for $i\ge 1$ let $s_{i}:=\min(2s_{i-1}, t)$. Let $k$ be the first integer for which $s_k=t$ and write 
    \begin{equation*}
         \Big| \frac{v_s}{s}-\frac{v_t}{t} \Big| \le  \Big| \frac{v_{s_0}}{s_0}-\frac{v_{s_1}}{s_1} \Big|+\cdots + \Big| \frac{v_{s_{k-1}}}{s_{k-1}}-\frac{v_{s_k}}{s_k} \Big| \le Cs_0^{-1/2+2\epsilon } +\cdots +Cs_k^{-1/2+2\epsilon }\le Cs^{-1/2+2\epsilon},
    \end{equation*}
    where the second inequality is by Corollary~\ref{cor:t/2}.
\end{proof}

We can now finish the proof of Proposition~\ref{prop:var}. Indeed, choose some $s$ such that the right-hand side of \eqref{eq:cor3} is smaller than 0.05 (take the smallest such $s$ and call it $s_0$). Then require that $a$ is sufficiently small so that for any $t\le s_0$ we have that $d\cdot v_t/t\in [0.95,1.05]$. This can be done since there is a coupling of $W$ with a simple random walk $X$ such that $\mathbb P (W(t)\neq X(t))\le Cta$ and therefore when $t\le s_0$ we have
\begin{equation*}
    |v_t-t/d| = \big|\mathbb E [W(t)_1^2-X(t)_1^2]\big| \le 2t^2\mathbb P (W(t)\neq X(t)) \le Cs_0^2a.
\end{equation*}
If $T'\le s_0$, the inductive assumption in \eqref{eq:induction} is proved immediately (this is the induction base). If $T'>s_0$ then it follows from Corollary~\ref{cor:3}, used with $s=s_0$ and $t=T'$ that $d\cdot v_{T'}/T'\in [0.9,1.1]$. This completes the induction in \eqref{eq:induction}. It follows that Corollary~\ref{cor:3} holds with $T'=T$ which implies Proposition~\ref{prop:var}.

\subsection{Traveling far}

\begin{lem}\label{lem:traveling}
    For any $t\le T$ we have that 
    \begin{equation*}
        \mathbb P \Big( \max _{s\le t} \|W(s)\| \ge t^{1/2+3\epsilon } \Big) \le Ce^{-c\log ^2t}.
    \end{equation*}
\end{lem}

\begin{proof}
    Fix $t\le T$ sufficiently large. Let $W_1,\dots ,W_k$ be independent ORRW of length $\lfloor \sqrt{t} \rfloor $ with $k$ being $\lceil t/\lfloor \sqrt{t}\rfloor  \rceil \approx \sqrt{t}$. Let $\tilde W$ be the concatenation of $W_1,\dots ,W_k$. By Corollary~\ref{cor:concat}, there is a coupling of $W$ and $\tilde W$ such that 
\begin{equation}\label{eq:concat6}
 \mathbb{P} \Big( \max _{s\le t}\|W(s)-\tilde W (s)\|\ge 2t^{1/2+2\epsilon } \Big) \le Ce^{-c\log^{2}t}.
\end{equation}
Next, for any $s\le t$, each coordinate of $\tilde{W}(s)$ is a sum of at most $k$ independent variables with zero expectation whose variance is bounded by $\sqrt{t}$ by Proposition~\ref{prop:var}, that are deterministically bounded by $\sqrt{t}$. Thus, by Freedman's inequality \cite[Theorems 3.6 and 3.7 (8 \& 9 in the arXiv version)]{chung2006concentration} and a union bound we have that 
\begin{equation*}
    \mathbb P \Big(  \max _{s\le t} \|\tilde W (s)\| \ge t^{1/2+2\epsilon }  \Big)\le Ct\exp \Big( -\frac{ct^{1+4\epsilon }}{k\sqrt{t} +t^{1+2\epsilon } }\Big) \le Ce^{-\log ^2t}.
\end{equation*}
Combining the last estimate with \eqref{eq:concat6} completes the proof of the lemma.
\end{proof}

\subsection{Coupling with Brownian motion}

Recall that $\sigma  =\sqrt{v_T/T}$ and $\sigma ^2\in [0.9/d,1.1/d]$.

\begin{lem}\label{lem:coupling}
For all $t\le T$, there is a coupling of $W$ and a standard Brownian motion $B$ such that
\[
\mathbb{P} \Big( \max _{s\le t}\|W(s)-\sigma B(s)\|\ge t^{1/3+4\epsilon } \Big) \le Ce^{-c\log^{2}t}.
\]
\end{lem}

\begin{proof}
Let $W_1,\dots, W_k$ be independent ORRW of length $t':=\lfloor t^{2/3} \rfloor $ with $k$ being $\lceil t/t' \rceil \approx t^{1/3}$. Let $\tilde W$ be the concatenation of $W_1,\dots , W_k$ and let $W$ be an ORRW of length $t$. Define the random variables 
\begin{equation*}
    X_i:=W_i(t') \mathds 1 \{\|W_i(t')\| \le t^{1/3+2\epsilon }\} \quad \text{and} \quad Y_i:=X_i/{\rm Std}((X_i)_1) .
\end{equation*}
By the symmetries of the lattice, $\mathbb E [Y_i]=0$ and the covariance matrix of $Y_i$ is the identity. Let $S_i:=\sum _{j=1}^{i}Y_j$. By Zaitsev's version of KMT \cite[Theorem~1.3]{zaitsev1998multidimensional} using that $\|Y_i\|\le t^{2\epsilon }$ we obtain that there is a standard Brownian motion $B'$ in $\mathbb R^d$ such that  
\begin{equation*}
    \mathbb P \Big( \max _{i\le k} \|S_i-B'(i)\| \ge t^{3\epsilon }   \Big) \le Ce^{-ct^\epsilon}.
\end{equation*}
To aid the reader in understanding Zaitsev's setup, let us mention that he defines a family ${\mathcal A}_d(\tau)$ of distributions \cite[Definition 1.2]{zaitsev1998multidimensional}. It is straightforward to check that a variable bounded by $M$ belongs to ${\mathcal A}_d(C(d)M)$. Zaitsev's $L(x)$ is $\max(1,\log x)$ and his $\Delta$ is defined on the first page of \cite{zaitsev1998multidimensional}.

Let $B$ be the standard Brownian motion defined by the scaling $B(s):=\sqrt{t'}B'(s/t')$ and let $\sigma ':= {\rm Std}(X_i)_1/\sqrt{t'}$. Then
\[
\mathbb P \bigg( \max _{i\le k} \Big\| \sum _{j\le i} X_j-\sigma ' B(it') \Big\| \ge {\rm Std}(X_i)t^{3\epsilon }   \bigg) \le Ce^{-ct^\epsilon}.
\]
and since ${\rm Std}(X_i)\le 1.1t^{1/3}/\sqrt{d}<t^{1/3}$ (Proposition~\ref{prop:var}) we can conclude that 
\begin{equation}\label{eq:0}
    \mathbb P \bigg( \max _{i\le k} \Big\| \sum _{j\le i} X_j-\sigma ' B(it') \Big\| \ge t^{1/3+3\epsilon }   \bigg) \le Ce^{-c\log ^2t}.
\end{equation}
Let $\Omega $ be the event that for all $i\le k$ we have that 
\begin{equation*}
    \max _{s\le t'} \|W_i(s)\| \le t^{1/3+2\epsilon } \quad \text{and} \quad \max _{s\in [(i-1)t',it']} \|B(s)-B((i-1)t')\|\le t^{1/3+2\epsilon }. 
\end{equation*}
By Lemma~\ref{lem:traveling} and a standard Brownian motion estimate we have that $\mathbb P (\Omega ) \ge 1-Ce^{-c\log ^2t}$. On the event $\Omega $ we have that $\sum _{j=1}^iX_j=\tilde W (it')$ for any $i\le k$ and therefore on $\Omega $ and the complement of the event in \eqref{eq:0} we have that $\max _{s\le t}\|\tilde W(s)-\sigma 'B(s)\| \le 2t^{1/3+3\epsilon }$. Thus, using also Corollary~\ref{cor:concat} we obtain that 
\begin{equation}\label{eq:5}
    \mathbb{P} \Big( \max _{s\le t}\|W(s)-\sigma ' B(s)\|\ge 2t^{1/3+3\epsilon } \Big) \le Ce^{-c\log^{2}t}.
\end{equation}
We would like to replace $\sigma '$ with $\sigma =\sqrt{v_T/T}$. We have that 
\begin{equation*}
    |\sigma -\sigma '| \le \big| \sigma - \sqrt{v_{t'} /t'} \big| + \big| \sqrt{v_{t'} /t'}  -\sigma ' \big|.
\end{equation*}
By Proposition \ref{prop:var},
\[
\Big|\frac{v_{t'}}{t'}-\sigma^2\Big|\le C(t')^{-1/2+5\epsilon}\le C t^{-1/3 +4\epsilon }
\]
so $|\sigma-\sqrt{v_{t'}/t'}|\le Ct^{-1/3+4\epsilon}$. As for the second term, by the fact that $X_i=W_i(t')$ with probability at least $1-Ce^{-c\log ^2t}$ we get that $|\sqrt{v_{t'}/t'}-\sigma'|<Ce^{-c\log^2t}$. Altogether we get $|\sigma-\sigma'|\le Ct^{-1/3+4\epsilon}$. Since $\max _{s\le t}\|B(s)\|\le t^{1/2+\epsilon }$ with very high probability, we may safely replace $\sigma '$ with $\sigma $ in equation \eqref{eq:5} to complete the proof of the lemma.
\end{proof}

As in Corollary~\ref{cor:3.2}, we can get a coupling with Brownian motion that becomes better in the beginning at the expense of the probability bound.

\begin{lem}\label{lem:3.3}
    For any $t\le T$, there is a coupling of $W$ and a standard Brownian motion $B$ such that 
    \begin{equation*}
        \mathbb P \big( \exists s\le T, \ \|W(s)-\sigma B(s)\| \ge (t+s)^{1/3+5\epsilon } \big) \le Ce^{-c\log ^2t}.
    \end{equation*}
\end{lem}

\begin{proof}
    Let $k$ be the first integer for which $t\sum _{i=1}^k2^{i-1}\ge T$. For any $i\in [1,k-1]$ let $t_i:= 2^{i-1}t$ and let $t_k:=T-t_1-\cdots -t_{k-1}$ be the remainder. Let $W_1,\dots ,W_k$ be independent ORRW of lengths $t_1,\dots ,t_k$ respectively and let $\tilde W$ be their concatenation of length $T$. In the natural coupling of $W$ and $\tilde W$, it follows by induction and Lemma~\ref{lem:concat} that for all $i\le k$ we have 
\begin{equation}\label{eq:3}
        \mathbb P \Big( \max _{s\le s_i} \|W(s)-\tilde W (s)\| \ge s_1^{2\epsilon }+\cdots +s_i^{2\epsilon }  \Big) \le \sum _{j=1}^i Ce^{-c\log ^2s_j} ,
    \end{equation}
    where $s_j:=t_1+\dots +t_j$. Next, let $B_1,\dots B_k$ be the independent Brownian motions from Lemma~\ref{lem:coupling} applied for the walks $W_1,\dots ,W_k$ respectively. The Brownian motion $B$ given by the concatenation of $B_1,\dots ,B_k$ satisfies for all $i\le k$
    \begin{equation*}
        \mathbb P \Big( \max _{s\le s_i} \|\tilde W (s)-\sigma B(s)\| \ge t_1^{1/3+4\epsilon }+\cdots +t_i^{1/3+4\epsilon }  \Big) \le \sum _{j=1}^i Ce^{-c\log ^2t_j}.
    \end{equation*}
    The lemma follows from this estimate together with \eqref{eq:3} and a union bound over $i\le k$.
\end{proof}

\subsection{Transition probabilities}

In the following lemma we establish Assumption~\hyperref[h1]{\bf{(H1)}} on transition probabilities at time $T$.

\begin{lem} \label{lem:H1}
For any $v\in \mathbb Z ^d$ we have that 
\begin{equation}\label{eq:transition}
    \mathbb P (W(T)=v)\le \min \big( T^{-d/2+\nu }, \|v\|^{-d+2\nu } \big).
\end{equation}
\end{lem}

\begin{proof}
Let $t_2:=\lfloor T^{4/5} \rfloor $ and $t_1:=T-t_2$. Let $W_1,W_2$ be independent ORRW of lengths $t_1,t_2$ respectively and let $\tilde W$ be their concatenation of length $T$. Let $R:=T^{2/5+3\epsilon }$ and let $u\in \mathbb Z^d$. Summing over the values of $w=W_2(t_2)$ at a reasonable distance from $u$ we obtain
\begin{equation*}
\begin{split}
    \mathbb P \big( &\tilde W(T)=u \big) \le \mathbb P \big( \|W_2(t_2)\|\ge R \big) + \sum _{\|w\|\le R} \mathbb P (W_1(t_1)=u-w) \cdot \mathbb P (W_2(t_2)=w) \\
    &\le Ce^{-c\log ^2T} + t_2^{-d/2+\nu } \mathbb P \big( \|W_1(t_1)-u\| \le R \big) \le Ct_2^{-d/2+\nu } R^d T^{-d/2}\le CT^{-d/2+0.8\nu 
 +3d\epsilon },
\end{split}
\end{equation*}
where in the second inequality we used Lemma~\ref{lem:traveling} to bound the first term and Assumption~\hyperref[h1]{\bf{(H1)}} at time $t_2$ to bound the second term. The third inequality follows from the Brownian approximation since $\mathbb P ( \|W_1(t_1)-u\| \le R ) \le \mathbb P ( \|\sigma B(t_1)-u\| \le 2R ) +Ce^{-c\log ^2T} \le CR^dT^{-d/2}$, where $B$ is the Brownian motion from Lemma~\ref{lem:coupling} with $t=t_1$. 

Next, let $r:=T^{2\epsilon }$ and let $\Omega := \{\max _{s\le T}\|W(s)-\tilde W(s)\| \le r\}$ be the event that the concatenation works. Using Lemma~\ref{lem:concat} and the bound on  $\mathbb P ( \tilde W(T)=u )$ we obtain
\begin{equation*}
    \mathbb P (W(T)=v) \le \mathbb P (\Omega ^c) +\mathbb P \big( \| \tilde W(T)-v\|\le r \big) \le e^{-c\log ^2T} +Cr^dT^{-d/2+0.8\nu 
 +3d\epsilon }\le T^{-d/2+0.9\nu },
\end{equation*}
where in the last inequality we used that $\epsilon <\nu /(50d)$ and that $T$ is sufficiently large (recall that $T>a^{-1/(2d)}$ and $a$ is sufficiently small). This finishes the proof of \eqref{eq:transition} when $\|v\|\le T^{1/2+\nu /(20d)}$ since then $T^{-d/2+0.9\nu }\le \|v\|^{-d+2\nu }$. When $\|v\|\in [T^{1/2+\nu /(20d)},T]$, we can use Lemma~\ref{lem:traveling} (since $\epsilon <\nu /(60d)$) to obtain
\[
\mathbb P(W(T)=v)\le\mathbb P(\|W(T)\|\ge T^{1/2+3\epsilon})\le Ce^{-c\log^2T}
\le \min(T^{-d/2+\nu},\|v\|^{-d+2\nu}),
\]
as long as $T$ is large enough. Finally, when $\|v\| >T$ the bound \eqref{eq:transition} holds trivially since $\mathbb P (W(T)=v)=0$.
\end{proof}

\subsection{Escaping into a ball}

\begin{lem}\label{lem:3.4}
    Let $1<L\le \sqrt{T}$ and let $D(0,L):=\{y\in\mathbb Z ^d: \|y\|_2\le L\}$ be the Euclidean ball around $0$. Let $x\in \mathbb Z^d$ be a vertex in the outer vertex boundary of $D(0,L)$ and let $W_x$ be ORRW starting from $x$. Then
    \begin{equation*}
        \mathbb P \big( W_x [1,L^2]\subseteq D(0,L)\text{ and } \|W_x(L^2)\|_2 \le L/2 \big) \ge cL^{-1.01}.
    \end{equation*}
\end{lem}

\begin{proof}
Let $\lambda$ be some sufficiently large parameter to be fixed later.
We may assume that $L\ge L_{0}(\lambda)$ i.e.\ that $L$ is sufficiently
large as a function of $\lambda$, since the probability in Lemma
\ref{lem:3.4} is positive for $L>1$. Let $y\in D(0,L-\lambda)$ with $\|y-x\|_2 \le 2\lambda$.
Let $\gamma$ be the shortest path in $D(0,L)$ from $x$ to $y$
(except $\gamma(0)=x$ is not in $D(0,L)$). Of course, $\len\gamma\le C\lambda$. 

Let $\mathcal{A}:=\{\forall t\le \len \gamma, \ W_x(t)=\gamma(t)\}$ be the event that $W_{x}$ has its first $\len\gamma$
steps exactly following $\gamma$. Note that $\mathbb P (\mathcal A)\ge c_\lambda$. Let $(\bar W(t))_{t\ge0}$ be $(W_{x}(\len\gamma+t))_{t\ge0}$
conditioned on $\mathcal{A}$. Equivalently one may consider $\bar W$ not as conditioned ORRW
but as ORRW starting from $y$ and evolving in a non-virgin environment in which the edges of $\gamma$ are reinforced, and this is the point of view we will adopt. Let $W_y$ be an auxiliary ORRW starting
from $y$ in a virgin environment.  Consider the natural coupling between $\bar W[0,L^{2}-\len\gamma]$
and $W_y[0,L^{2}-\len\gamma]$.  

Let $r_{0},\dotsc,r_{k}$ satisfy that $r_{0}=\lambda$, $r_{i+1}\in[r_{i}^{1.1},r_{i}^{1.2}]$ and $r_{k}=\sqrt{\frac{1}{2}(L^{2}-\len\gamma)}$ (clearly,
such a sequence may be constructed if $\lambda$ is sufficiently large
and $L$ is sufficiently large depending on $\lambda$). For $i\ge1$, let $t_{i}:=\lfloor 2r_{i}^{2}\rfloor $ and let $\mathcal{G}_{i}:=\{\forall t\le t_i, \ \bar W (t)=W_y(t)\}$ be the event that the natural coupling
of $\bar W$ with $W_y$ succeeded up to time $t_{i}$. In other words, $\mathcal{G}_{i}$ is the event that
for every $t<t_i$ for which $W_y(t)\in\gamma$, the variable $U_t$ from Definition~\ref{def:natural} implies that the next step is identical: $W_y(t+1)=\bar W(t+1)$.

Define next $\mathcal{S}_{i}$ to be the event that $\mathcal{G}_{i}$
occurred, that $W_y[0,t_{i}]\subseteq D(0,L)$ and that $W_y(t_{i})\in D(0,L-r_{i})$.
For the proof of the lemma it is enough to lower bound $\mathbb{P}(\mathcal{S}_{k})$ since on $\mathcal S _k$ we have $\bar W[0,L^2-\len \gamma ]\subseteq D(0,L)$ and $\bar W(L^2-\len \gamma)\in D(0,L/2)$. We will do that by induction
on $i$, with a small difference in the proof of $\mathcal{S}_{1}$.

Let $W_{i}$ be an ORRW independent of $W_y[0,t_{i-1}]$ of length
$t_{i}-t_{i-1}$ starting from the origin (for $i=1$ define $t_{0}=0$
and remove the independence requirement). By Lemma~\ref{lem:coupling}
there is a coupling of $W_{i}$ and a Brownian motion $B$ with diffusive
constant $\sigma$, such that the event $\mathcal{B}_{i}:=\big\{\forall s\le t_{i}-t_{i-1},\ \|W_{i}(s)-B(s)\|\le t_{i}^{1/3+4\epsilon}\big\}$
satisfies $\mathbb{P}(\mathcal{B}_{i})\ge1-Ce^{-c\log^{2}t_{i}}$.
Here $B$ is also independent of $W_y[0,t_{i-1}]$. Define the event
$\mathcal E _{i}$ by 
\[
\big\{ W_y(t_{i-1})+B[0,t_{i}-t_{i-1}]\subseteq D(0,L-\tfrac{1}{2}r_{i-1})\text{ and }\|W_y(t_{i-1})+B(t_{i}-t_{i-1})\|_{2}\le L-\tfrac{5}{4}r_{i}\big\}.
\]
Assuming that $\mathcal{S}_{i-1}$ occurred we know that $W_y(t_{i-1})\in D(0,L-r_{i-1})$.
Inserting this into Lemma \ref{lem:brownian} with $R_{\textrm{Lemma \ref{lem:brownian}}}=L-\frac12r_{i-1}$, $\ell_{\textrm{Lemma \ref{lem:brownian}}}=\frac54 r_i$ and $M_{\textrm{Lemma \ref{lem:brownian}}}$ a sufficiently large constant gives $\mathbb{P}(\mathcal E _{i}\mid\mathcal{S}_{i-1})\ge cr_{i-1}/r_{i}$.

We can now estimate $\mathbb{P}(\mathcal{S}_{i})$. In the case $i=1$
we use Lemma~\ref{lem:avoiding} (with $t_{\textrm{Lemma~\ref{lem:avoiding}}}=\len\gamma$). As long as $a$ is sufficiently small depending on $\lambda $, the event $\mathcal{A}$ implies that the time $\len\gamma$
is relaxed for $W_{x}$. Thus, we get that the natural
coupling of $W_y[0,t_{1}]$ with $\bar W[0,t_{1}]$ succeeds with probability
at least $1-Ca^{1/9}$. By construction $t_1^{1/3+4\epsilon }\ll r_0$ for large $\lambda $ and therefore if in addition we have $\mathcal E _{1}$ and
$\mathcal{B}_{1}$, then $W_y[0,t_{1}]\subseteq D(0,L)$
and $W_y(t_{1})\in D(0,L-r_{1})$. Taking into consideration the probabilities
of these three events we get
\begin{equation}\label{eq:S_1}
    \mathbb{P}(\mathcal{S}_{1})\ge cr_{0}/r_{1}-Ca^{1/9}-Ce^{-c\log^{2}t_{1}}\ge cr_0/r_1,
\end{equation}
where the last inequality holds when $\lambda$ is sufficiently large and $a$ is sufficiently small.

If $i>1$ we apply Lemma~\ref{lem:concat} to concatenate $W_y[0,t_{i-1}]$
with $W_{i}$. Let $\tilde{W}$ be the concatenation of $W_y[0,t_{i-1}]$
and $W_{i}$ and let $\mathcal{C}_{i}:=\{\max_{s\le t_{i}}\|W_y(s)-\tilde{W}(s)\|\le t_{i}^{2\epsilon}\}$ be the event that $W_y$ and $W_{i}$ are concatenated successfully.
By Lemma~\ref{lem:concat} we have that $\mathbb{P}(\mathcal{C}_{i})\ge1-Ce^{-c\log^{2}t_{i}}$.
We claim that $\mathcal{S}_{i-1}\cap\mathcal{B}_{i}\cap\mathcal{C}_{i}\cap \mathcal E _{i}\subseteq \mathcal{S}_{i}$. Indeed, by construction $t_i^{1/3+4\epsilon }+t_i^{2\epsilon }\ll r_{i-1}$ for large $\lambda $ and therefore the intersection of events
implies that $W_y[t_{i-1},t_{i}]\subseteq D(0,L-2\lambda)$ and $W_y(t_{i})\in D(0, L-r_{i})$. Second, the condition in $\mathcal{S}_{i-1}$
that $\bar W(t)=W_y(t)$ for all $t\le t_{i-1}$ continues to hold in the time interval $[t_{i-1},t_i]$ because $W_y[t_{i-1},t_i]\subseteq D(0,L-2\lambda)$ so $\bar W$ does not interact with the edges in $\gamma $. Thus we get $\mathcal{S}_{i}$.
We conclude that 
\[
\mathbb{P}(\mathcal{S}_{i})\ge\mathbb{P}(\mathcal{S}_{i-1} \cap \mathcal E _i)-\mathbb{P}(\mathcal{B}_{i}^c)-\mathbb{P}(\mathcal{C}_{i}^c)\ge c\frac{r_{i-1}}{r_{i}}\mathbb{P}(\mathcal{S}_{i-1})-Ce^{-c\log^{2}t_{i}}.
\]
It follows from \eqref{eq:S_1}, the last estimate, and induction that for all $i$ we have $\mathbb{P}(\mathcal{S}_{i})\ge c^{i}/r_{i}$,
if $\lambda$ is sufficiently large (this is the last requirement
on $\lambda$ and we may now fix its value). Since $k\approx\log\log L$, we obtain that $\mathbb P (\mathcal S _k) \ge cL^{-1.01}$. This finishes the proof of the lemma since $\bar W(\cdot )$ is $W_x(\len \gamma +\cdot  )$ conditioned on $\mathcal A $, and since $\mathbb P(\mathcal A )\ge c_\lambda$.
\end{proof}

\section{The capacity estimate and heavy blocks}\label{sec:heavy}

In this section and the next we prove Assumption~\hyperref[h2]{\bf{(H2)}} at time $T$. While this assumption is formulated for the demon walk $\W$, there is still some important preparation work to be done for usual ORRW. Recall that $R=\lfloor\sqrt{T}/2\rfloor$. 

\begin{definition}
Let $A$ be a subset of $[-4R,4R]^d$. We say that $A$ is nowhere heavy if for all $u \in [-2R,2R]^d$ and $r\in [R^{\delta },R]$ the box $\Lambda :=u+[-r,r]^d$ satisfies
\begin{equation}\label{eq:nh}
    |A\cap \Lambda| \le r^{\kappa}.
\end{equation}
\end{definition}
Notice that, despite the imposing name, $A$ might be as heavy as possible near the boundary of the box --- the definition only controls its behaviour in $[-3R,3R]^d$. This is in line with the fact that Assumption~\hyperref[h2]{\bf{(H2)}} is only for $[-R,R]^d$. The demon might create uncontrollably bad behaviour near the boundary, so we always focus on the central parts of the box.

As explained in the introduction (\S\ref{sec:sketch}), the proof requires a specialised notion of capacity. Let us now define this precisely. (See, e.g., \cite[Section~6.5]{lawler2010random} for background on the standard random walk capacity. We remark that, unlike the usual capacity, our notion does not have a variational formula, and is not necessarily monotone in $A$).

\begin{definition}\label{def:esc}
  Let $A\subseteq [-4R,4R]^d$, let $u\in [-2R,2R]^d$ and $r\in[R^\delta, R]$. Let $\Lambda^+=u+[-2r,2r]^d$. For $z\in A\cap \Lambda^+ $, let $W_z$ be a once-reinforced walk starting from $z$ and let $\Omega_{\rm esc}^{\Lambda^+ }(z,A)$ be the event that $W_z$ successfully escaped from $A$ in $\Lambda^+$ which means that: 
\begin{enumerate}
    \item 
    For all $t\in [1,2R^2]$ we have that $W_z(t)\in[-4R,4R]^d\setminus A$.
    \item 
    For all $t\in [R^2,2R^2]$ we have that $W_z(t)\in [-R,R]^d$.
    \item 
    For all $t\in [40r^2,2R^2]$ we have that $W_z(t)\notin \Lambda^+ $.
\end{enumerate}
Define the once-reinforced capacity of $A$ in $\Lambda^+ $ to be 
\begin{equation*}
   {\rm Cap}_{\Lambda^+} (A):= \sum _{z\in A\cap \Lambda^+ } \mathbb P \big( \Omega _{\rm esc}^{\Lambda^+} (z,A) \big).
\end{equation*}
\end{definition}

In the next lemma, we bound the volume of $A\cap \Lambda $ by the capacity. 

\begin{lem}[Capacity-Volume bound]\label{lem:Cap-Vol}
  Let $A\subseteq [-4R,4R]^d$ and suppose that $A$ is nowhere heavy. Let $\Lambda=u+[-r,r]^d$ and $\Lambda^+=u+[-2r,2r]^d$. Then ${\rm Cap}_{\Lambda^+} (A)\ge r^{-6\nu }|A\cap \Lambda |^{1-2/d }$.
\end{lem}

The idea of one of the classical proofs of the volume capacity bound is as follows. For any point $x$ in $A$, contributing to the volume of $A$, we can release a simple random walk $W_x$ from $x$ and wait until it exits $A$ for the last time. Starting from this time, we obtain a walk escaping the boundary of $A$, contributing to the capacity of $A$. Estimating the amount of double counting in this argument leads to the quantitative bound ${\rm Cap} (A) \ge c|A|^{1-2/d}$. Our proof below is a version of this argument for ORRW.

\begin{proof}
  Since $r\ge R^\delta$, $R=\lfloor \sqrt{T}/2\rfloor$ and $T>a^{-1/(2d)}$, we may assume $r$ is sufficiently large. For $x\in A\cap \Lambda $ let $W_x$ be a once-reinforced walk starting from $x$. Define the events
    \begin{equation*}
    \begin{split}
        &\Omega _1(x):= \big\{\forall t\in [\tfrac12r^2,3R^2] \text{ with }W_x(t)\in [-\tfrac{8}{3}R,\tfrac{8}{3}R]^d \text{ we have }  d(W_x(t),A)> t^{4\epsilon } \big\}\\
        &\Omega _2(x) :=  \bigg\{  \!\! \begin{array}{cc}
        W_x [0,\frac{1}{2}r^2]\subseteq u+[-\frac{3}{2}r,\frac{3}{2}r]^d, \  \forall t\in [30r^2,3R^2], \  d(W_x(t),\Lambda ^+)> t^{4\epsilon }  \\
         W_x[\frac{1}{2}r^2,3R^2]\subseteq [-\frac{8}{3}R,\frac{8}{3}R]^d, \  W_x[\frac{3}{4}R^2,3R^2]\subseteq [-\frac{11}{12}R,\frac{11}{12}R]^d
    \end{array} \!\! \bigg\}
    \end{split}
    \end{equation*}
and $\Omega (x):=\Omega _1(x)\cap \Omega _2(x)$ (see Figure \ref{fig:blueA}, left and middle panes). The idea is that on $\Omega (x)$ the walk $W_x$ is forced to exit $A$ in time $r^2$, remain far from $A\cup \Lambda $ and spend time in $[-R,R]^d$. On this event we will construct ORRW escaping the boundary of $A$ and satisfying $\Omega _{\rm esc}^{\Lambda^+}(z,A)$. 

\begin{figure}[htp]
    \centering
\includegraphics[width=16.5cm]{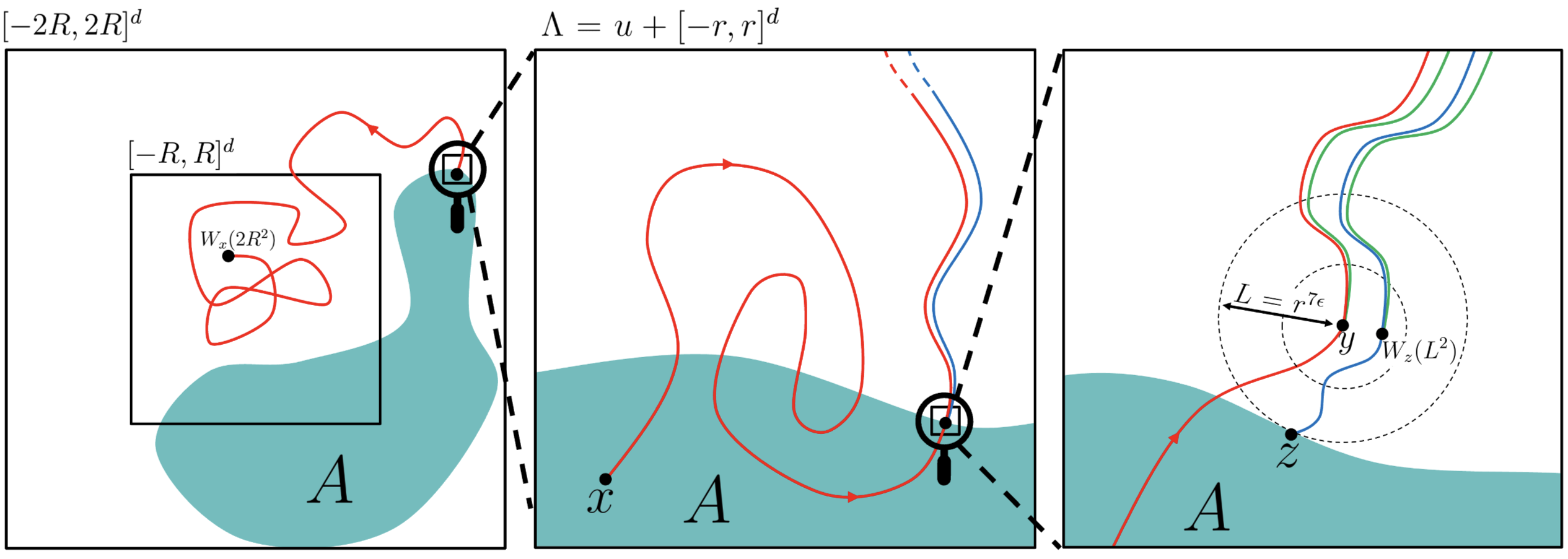}
\caption{The three scales in the proof of Lemma~\ref{lem:Cap-Vol}. The left pane shows the large scale ($R$): the walk $W_x$ starting from $x\in A\cap \Lambda $, exits $A\cup \Lambda $ quickly, doesn't get close again, enters the inner box $[-R,R]^d$, and spends $R^2$ time in there. The middle pane, at scale $r$, shows the last exit of the ORRW from $A$ and the path $W_z$ in blue. The right pane, at scale $r^{7\epsilon}$, shows the coupling in more detail: The red path is $W_x$ (and $W_1+x$, which is its section until $y$). The blue path is $W_z$ (and $W_1'+z$, which is its first section). The green paths are the two translations of $W_2$.
}\label{fig:blueA}
\end{figure}

We start by proving that $\mathbb P(\Omega (x))\ge c$. To bound the probability of $\Omega _1(x)$ we use that $A$ is nowhere heavy and Lemma~\ref{lem:A}. Let $k_0:=\lfloor \log _2(r/2) \rfloor $ and $k_1:=\lceil \log _2(2R) \rceil $. For any $k\in [k_0,k_1]$ define the set
\begin{equation}\label{eq:Ak}
    A_k:=\big\{ y\in [-\tfrac{8}{3}R,\tfrac{8}{3}R]^d : \exists z\in A \text{ with } \|z-y\|\le 2^{8\epsilon(k+1)} \big\}.
\end{equation}
Let now $L\ge 2^k$. Any $z$ as in \eqref{eq:Ak} must be in $[-3R,3R]^d$ for $R$ sufficiently large. Similarly, if the $y$ of \eqref{eq:Ak} is in $x+[-L,L]^d$ then $z$ must be in $u+[-7L,7L]^d$ because
\[
\|z-u\|_\infty\le\|z-y\|_\infty+\|y-x\|_\infty+\|x-u\|_\infty
\le 2^{8\eps (k+1)} + L +r \le 2L^{8\eps}+L+4\cdot2^{k_0}\le 7L.
\]
Hence
\begin{equation*}
  |A_k\cap (x+[-L,L]^d)|\le CL^{8\epsilon d}|A\cap [-3R,3R]^d \cap (u+[-7L,7L]^d) |
  \le L^{\kappa +\nu }.   
\end{equation*}
(The last inequality follows from \eqref{eq:nh}. We also used that $r$ is sufficiently large to get rid of the constant, and for $L\ge \frac17 R$ in fact applied \eqref{eq:nh} for a constant number of boxes and summed, since \eqref{eq:nh} only applies to boxes smaller than $R$). Thus, by Lemma~\ref{lem:A} with the walk $W_x-x$, the set $A_k-x$ and with $L_{\textrm{Lemma~\ref{lem:A}}}=2^k$ we have
\begin{equation}\label{eq:515}
\mathbb P \big( \exists t\in [4^k,T), \ W_x(t)\in A_k \big)\le C2^{-k/3}.
\end{equation}
Break the $t$ in the definition of $\Omega_1$ into scales, recall that $3R^2<T$  and get
\begin{equation}\label{eq:Omega_1}
  \mathbb P \big( \Omega _1(x)^c \big) \le \sum _{k=k_0}^{k_1} \mathbb P \big( \exists t\in [4^k,3R^2], \ W_x(t)\in A_k \big)
  \stackrel{\smash{\textrm{\eqref{eq:515}}}}{\le}
  \sum _{k=k_0}^{k_1} C2^{-k/3} \le Cr^{-1/3}.
\end{equation}
    Next, we use the Brownian approximation in order to show that $\mathbb P (\Omega _2(x))\ge c$. We use Lemma~\ref{lem:3.3} and get a coupling of $W_x$ with a Brownian motion $B$ with diffusive constant $\sigma ^2 \in [\frac{0.9}{d},\frac{1.1}{d}]$ starting from $x$ such that the event $\mathcal B :=\big\{\forall t\le 2R^2,\|W_x(t)-B(t)\|\le (t+r^2)^{1/3+5\epsilon  } \big\}$ holds with probability at least $1-e^{-c\log ^2r}$. Define the event
\begin{equation}\label{eq:brownian1}
        \Omega _2'(x) :=  \bigg\{  \!\! \begin{array}{cc}
        B [0,\frac{1}{2}r^2]\subseteq u+[-\frac{5}{4}r,\frac{5}{4}r]^d, \ \forall t\in [30r^2,3R^2], \  d(B(t),\Lambda ^+)> t^{2/5} \\ B[\frac{1}{2}r^2,3R^2]\subseteq [-\frac{7}{3}R,\frac{7}{3}R]^d, \  B[\frac{3}{4}R^2,3R^2]\subseteq [-\frac{10}{12}R,\frac{10}{12}R]^d 
    \end{array} \!\! \bigg\}.
    \end{equation}
A straightforward Brownian motion estimate shows that $\mathbb P (\Omega _2'(x) )\ge c$. For completeness we prove this in the appendix (Lemma \ref{lem:Om2p}). A simple check shows that $\mathcal B \cap \Omega _2'(x)\subseteq \Omega _2(x)$ for $r$ sufficiently large, completing the proof that $\mathbb P (\Omega _2(x))\ge c$. Combining this with \eqref{eq:Omega_1} finishes the proof that $\mathbb P (\Omega (x))\ge c$, since $r$ is sufficiently large.

Summing over $x\in A\cap \Lambda $ we obtain
\begin{equation*}
        |A\cap \Lambda | \le C\sum _{x\in A\cap \Lambda } \mathbb P (\Omega (x)).
\end{equation*}
We now wish to relate this last sum to the capacity. For this purpose, let $L:=r^{7\epsilon }$ and let $\mathcal T:=1+\max \{t\le 3R^2: d(W_x(t),A)\le L\}$. Since $\Omega(x)$ `prohibits' $W_x$ from getting too close to $A$ for large $t$, the maximum in the definition of $\mathcal T$ must be attained at small $t$. In other words, under $\Omega(x)$ we have 
    that  $\mathcal T\in [L,r^2/2]$ and $W_x(\mathcal T)\in S_L $ where
\[
S_L:= \big\{ y\in u+ [-\tfrac{3}{2}r,\tfrac{3}{2}r]^d : d(y,A)>L \text{ and }\exists y'\sim y, \ d(y',A)\le L  \big\}
\]
(see Figure \ref{fig:blueA}, right pane). We obtain that
\begin{equation}\label{eq:71}
    |A\cap \Lambda |\le C\sum _{x\in A\cap \Lambda } \sum _{t=1}^{r^2/2} \sum _{y\in S_L} \mathbb P \big( \Omega (x), \ \mathcal T=t, \ W_x(t)=y \big).
\end{equation}
Next, we fix $x\in A\cap \Lambda $, $t\in [L,r^2/2]$ and $y\in S_L$ and bound the last probability. Since $y\in S_L$, there exists some $z\in A\cap \Lambda ^+$ that is in the outer boundary of the Euclidean ball $D(y,L)$. Let $W_z$ be a once-reinforced walk starting from $z$ and let $\Omega _{\rm esc}^{\Lambda^+} (z,A)$ be the escape event defined above for the walk $W_z$. We couple $W_x$ and $W_z$ as follows (see Figure \ref{fig:blueA}, right pane). Let $W_1$, $W_2$ and $W_1'$  be three independent ORRW starting from the origin of lengths $t$, $2R^2$ and $L^2$ respectively. By Corollary~\ref{cor:3.2} the walk $W_x$ can be coupled with $W_1,W_2$ such that $W_x[0,t]-x=W_1[0,t]$ deterministically and the event
\begin{equation*}
    \mathcal C: =\Big\{ \forall s\le 2R^2, \ \|W_x(t+s)-W_x(t)-W_2(s)\| \le (t+s)^{3\epsilon } \Big\}.
\end{equation*}
satisfies $\mathbb P (\mathcal C ^c)\le Ce^{-c\log ^2t}\le e^{-c\log ^2T}$, where in the last inequality we used that $T\ge a^{-1/(2d)}$ is sufficiently large and $t\ge L=r^{7\epsilon }\ge T^c$. Similarly, we can couple $W_z$ with $W_1',W_2$ such that deterministically $W_z[0,L^2]-z=W_1'[0,L^2]$ and the event
\begin{equation*}
    \mathcal C': =\Big\{ \forall s\le 2R^2, \ \|W_z(L^2+s)-W_z(L^2)-W_2(s)\| \le (L^2+s)^{3\epsilon } \Big\}.
\end{equation*}
satisfies $\mathbb P (\mathcal C')\ge 1-e^{-c\log ^2T}$. Finally, by Lemma~\ref{lem:3.4} we have that $\mathbb P (\mathcal D) \ge r^{-8\epsilon }$ where 
\begin{equation*}
    \mathcal D := \big\{ W_z[1,L^2]\subseteq D(y,L) \text{ and }  \|W_z(L^2)-y\|\le L/2 \big\}.
\end{equation*}

We claim that 
\begin{equation}\label{eq:subset}
    \Omega (x) \cap \big\{ \mathcal T=t, W_x(t)=y \big\} \cap \mathcal C \cap \mathcal C' \cap \mathcal D \subseteq  \{W_1(t)=y-x\} \cap \Omega _{\rm esc}^{\Lambda^+} (z,A).
\end{equation}
To prove this, suppose that the intersection on the left hand side of \eqref{eq:subset} holds. In this case, the event $\{W_1(t)=y-x\}$ clearly holds and so it suffices to show that $\Omega _{\rm esc}^{\Lambda^+} (z,A)$ holds. On this intersection we have for all $s\le 2R^2$ that
\begin{equation}\label{eq:W_zW_x}
\begin{split}
    \|W_x(t+s)-W_z(L^2+s)\| &\le \|W_x(t)-W_z(L^2)\| +(t+s)^{3\epsilon } +(L^2+s)^{3\epsilon } \\
    &\le L/2+2(r^2+s)^{3\epsilon }. 
\end{split}
\end{equation} 
Now, on $\Omega (x)\cap \{\mathcal T=t\}$ we have $d(W_x(t+s),A)> L$ for all $s\in [0,2R^2]$ and $d(W_x(s),A)> s^{4\epsilon }$ for all $s\in[r^2/2,3R^2]$. It follows that $W_z(L^2+s)\notin A$ for all $s\le 2R^2$, for $r$ sufficiently large. By the definition of $\mathcal D$ we also have that $W_z(s)\notin A$ for $s\in [1,L^2]$.

The proof that $W_z$ avoids $\Lambda^+$ is similar. By the definition of $\Omega_2(x)$ we get that $d(W_x(s),\Lambda^+)\linebreak[0] >s^{4\epsilon}$ for all $s\in [30r^2,3R^2]$. With \eqref{eq:W_zW_x} we get
\begin{equation}\label{eq:WzLam}
d(W_z(s),\Lambda^+)>(s-L^2+t)^{4\epsilon}-L/2-2(r^2+s)^{3\epsilon},
\end{equation}
for $s\in[30r^2+L^2-t,3R^2+L^2-t]$. Under $\Omega(x)$ we have $t\in[L,r^2/2]$ and hence for $r$ sufficiently large this interval contains $[40r^2,2R^2]$ and the right-hand side of \eqref{eq:WzLam} is positive. Thus $W_z(s)\notin\Lambda^+$ for $s\in[40r^2,2R^2]$. A similar argument concludes from $W_x[0,3R^2]\subseteq[-\frac72 R,\frac72 R]^d$ that $W_z[0,2R^2]\subseteq [-4R,4R]^d$ (the claim on $W_x[0,3R^2]$ follows from the conditions on the first two time intervals in the definition of $\Omega_2(x)$ and from $u\in [-2R,2R]^d$). Finally, we conclude from $W_x[\frac34R^2,3R^2]\subseteq [-\frac{11}{12}R,\frac{11}{12}R]^d$ that $W_z[R^2,\linebreak[0]2R^2]\subseteq [-R,R]^d$. This finishes the proof of  \eqref{eq:subset}.

Using  \eqref{eq:subset} we obtain that 
\begin{equation*}
\begin{split}
  &\mathbb P \big( \Omega (x), \ \mathcal T=t, \ W_x(t)=y \big) \cdot \mathbb P (\mathcal D)- e^{-c\log ^2T}
  \le  \mathbb P \big( W_1(t)=y-x, \ \Omega _{\rm esc}^{\Lambda ^+} (z,A) \big) \\
  &= \mathbb P \big( W_1(t)=y-x \big) \cdot \mathbb P (\Omega _{\rm esc}^{\Lambda ^+} (z,A))
  \le  e^{-c\log ^2T} +\mathds 1_{\{\|x-y\| \le t^{1/2+3\epsilon  }\}} \cdot  t^{-d/2+\nu } \cdot \mathbb P (\Omega _{\rm esc}^{\Lambda ^+} (z,A)),
\end{split}
\end{equation*}
where in the first inequality we also used the bounds on the probabilities of $\mathcal C,\mathcal C'$ and the fact that $W_1'$ (and therefore $\mathcal D$) is independent of $W_x$. In the equality we used that $W_1$ is independent of $W_z$. Finally, in the last inequality we used Assumption~\hyperref[h1]{\bf{(H1)}} on transition probabilities together with Lemma~\ref{lem:traveling}.

Continuing, using the bound $\PP(\mathcal D)\ge r^{-8\epsilon}$ and $t\le \frac12r^2$ we obtain that 
\begin{equation}\label{eq:OmeTy_onex}
  \mathbb P \big( \Omega (x), \ \mathcal T=t, \ W_x(t)=y \big) \le
  e^{-c\log ^2T} +r^{3\nu } \mathds 1_{\{\|x-y\| \le t^{1/2+3\epsilon }\}} \cdot  t^{-d/2} \cdot \mathbb P (\Omega _{\rm esc}^{\Lambda^+} (z,A)).
\end{equation}
Next, note that
\[
\sum_{x\in A\cap \Lambda}\mathds 1_{\{\|x-y\| \le t^{1/2+3\epsilon }\}}t^{-d/2}
\le C\min\{t^{3d\epsilon}, |A\cap \Lambda|t^{-d/2}\},
\]
where the bound $t^{3d\epsilon}$ comes simply from the condition $\|x-y\| \le t^{1/2+3\epsilon }$. Thus, summing \eqref{eq:OmeTy_onex} over $x\in A\cap \Lambda $ we get
\begin{equation*}
  \sum _{x\in A\cap \Lambda } \mathbb P \big( \Omega (x), \ \mathcal T=t, \ W_x(t)=y \big)
  \le e^{-c\log ^2T} +r^{4\nu } \cdot \min \big( 1,|A\cap \Lambda |t^{-d/2} \big) \cdot \mathbb P (\Omega _{\rm esc}^{\Lambda^+} (z,A)).
\end{equation*}
Substituting this bound into \eqref{eq:71} we obtain
\begin{equation}\label{eq:72}
\begin{split}
  |A\cap \Lambda |&\le e^{-c\log ^2T} + C r^{4\nu }
  \sum _{t=1}^{r^2/2} \min \big( 1,|A\cap \Lambda  |t^{-d/2} \big)
  \sum _{ y\in S_L} \mathbb P (\Omega _{\rm esc}^{\Lambda^+} (z,A))  \\
    &\le e^{-c\log ^2T} +  r^{5\nu } {\rm Cap}_{\Lambda^+ }(A) \sum _{t=1}^{r^2/2} \min \big( 1,|A\cap \Lambda 
    |t^{-d/2} \big) 
\end{split}
\end{equation}
where in the last inequality we used that every $z\in A\cap \Lambda ^+$ is counted at most $CL^d$ times for every $y\in S_L$. The sum on the right hand side of \eqref{eq:72}, even when extended to infinity, is bounded by $C|A\cap \Lambda |^{2/d}$, which completes the proof of the lemma.
\end{proof}

At this point we switch from ORRW to the demon walk. Let, therefore, $\demon$ be some demon living outside $[-4R,4R]^d$ and let $\W$ be the demon walk. 

\begin{definition*}
For a block $\Lambda =u +[-r,r]^d$ with $u \in [-2R,2R]^d$ and $r\in [R^{\delta },R]$, define the stopping time
\begin{equation*}
    \tau _{\rm heavy}(\Lambda ) :=\min \big\{ t>0 : |\{\W(s): s\le t\}\cap \Lambda | \ge r^{\kappa } \big\}.
\end{equation*}
Let $\tau _{\rm heavy}$ be the first time $\tau _{\rm heavy}(\Lambda )$ happens for one of these $\Lambda $:
\begin{equation*}
     \tau _{\rm heavy} := \min \big\{ \tau _{\rm heavy} (\Lambda ): \Lambda =u +[-r,r]^d, \ u \in [-2R,2R]^d,\ r\in [R^{\delta },R] \big\}. 
\end{equation*}
\end{definition*}

The following theorem is the main result in this section.

\begin{thm}[No heavy blocks]\label{thm:heavy}
    We have that $\mathbb P \big( \tau _{\rm heavy} \le \inner  \big) \le e^{-c\log ^3T}$.
\end{thm}

From now on, throughout this section, we fix $\Lambda :=u+[-r,r]^d$ with $u \in [-2R,2R]^d$ and $r\in [R^{\delta },R]$ and prove the following proposition. Theorem~\ref{thm:heavy} follows immediately from this proposition and a union bound.

\begin{prop}\label{prop:heavy}
   We have that $\mathbb P\big( \tau _{\rm heavy} (\Lambda ) \le \min (\tau _{\rm heavy}, \inner ) \big) \le e^{-c\log ^3T}$. 
\end{prop}

The main idea in the proof of Proposition~\ref{prop:heavy} is to bound the capacity of the trajectory $\W[0,t]$ within $\Lambda $, and then translate this estimate to a volume bound using Lemma \ref{lem:Cap-Vol}.

For a time $t\ge 0$, let $A_t:=\W[0,t]$ be the history of $\W$ up to $t$. To prove Proposition~\ref{prop:heavy}, we would like to bound the volume of $A_\tau \cap \Lambda $ where 
\begin{equation*}
    \tau :=\min(\tau _{\rm heavy}, \inner) .
\end{equation*}
The history $A_\tau $ is nowhere heavy since $\tau \le \tau _{\rm heavy}$ and so by Lemma~\ref{lem:Cap-Vol} it suffices to bound its capacity. To this end, we need a demon version of the escape event (recall Definition~\ref{def:esc}).
\begin{definition*} For any $t$, define
  \[
  \ED(t)=\big\{ \W(t)\in\Lambda^+, \big( \W(t+s) \big) _{s\ge 0}\in\Omega_{\rm esc}^{\Lambda^+}(\W(t),\W[0,t]) \big\}.
  \]
\end{definition*}
In words, $\ED(t)$ is the event that $\W(t)\in \Lambda ^+$ and also the walk $(\W(t+s))_{s\ge 0}$ satisfies the three conditions in the definition of $\Omega _{\rm esc}^{\Lambda^+ }(z,A)$ with $z=\W(t)$ and $A=\W[0,t]$. We recall that this means that $\W$ avoids its history, escapes fast to $[-R,R]^d$ and spends some time there towards $\inner$.

Since $\Omega_{\rm esc}^{\Lambda^+}(z,A)$ contains the event that $W[1,2R^2]\cap A=\emptyset$ we see that the only interaction of $\W$ with its past is at time $t$. Hence
\begin{equation}\label{eq:escesc}
\mathbb P \big( \Omega_{\rm esc}^{\Lambda^+}(\W(t),\W[0,t]) \mid \mathcal F _t\big) \approx \mathbb P(\ED(t) \mid \mathcal F _t),
\end{equation}
where $\mathcal F _t :=\sigma (\W(0),\dots ,\W(t))$. In words, the left hand side is the probability that ORRW (in virgin environment) $W$ starting from $\W(t)$ satisfies $\Omega_{\rm esc}^{\Lambda^+}$ with respect to the set $\W[0,t]$, while the right hand side refers to a demon random walk in an existing environment. This also uses that, again under $\Omega_{\rm esc}^{\Lambda^+}(\W(t),\W[0,t])$, we have that $W$ does not leave $[-4R,4R]^d$ and hence in the relevant time period $[t,t+2R^2]$ there is no difference between $W$ and $\W$. 

For $p\in (0,1)$, we let $N_p$ be the number of excursions of $\W$ of length $r^2$ up to $\tau $ during which we had a possible such escape $\ED(t)$ with probability at least $p$:
\begin{equation*}
    N_p:=\big| \big\{i\le \lceil\tau /r^2\rceil : \exists t\in [ir^2,(i+1)r^2) \text{ with }\W(t)\in \Lambda^+  \text{ and }  \mathbb P  (\ED(t) \mid  \mathcal F _t  ) \ge p\big\} \big|.
\end{equation*}

\begin{lem}\label{lem:Np}
    For any $p\in (0,1)$ we have that 
\begin{equation*}
    \mathbb P \big( N_p \ge (\log ^4T)/p  \big) \le e^{-c\log ^3T}.
\end{equation*}
\end{lem}

\begin{proof}
Define a sequence of stopping times by $\zeta _0=0$ and inductively
\begin{equation*}
    \zeta _{j+1} := \min \big\{ t\ge \zeta _j +40r^2 : \W(t)\in \Lambda^+ \text{ and } \mathbb P  ( \ED(t) \mid \mathcal F _t) \ge p  \big\}.
\end{equation*}
First, observe that if $N_p\ge (\log ^4T)/p$ then at least $j_0:=\lfloor (\log ^4T)/(40p) \rfloor$ of these stopping times occur before $\tau $. Next, if the event $\ED(\zeta _{j})$ holds, then the walk $\W$ leaves $\Lambda^+$ within time $40r^2$ and does not return before time $2R^2$. Hence, if $\ED(\zeta _{j})$ holds, then walk spends at least $R^2$ time in $[-R,R]^d$ during the time interval $[\zeta _j,\zeta _j+2R^2]\subseteq [\zeta _j,\zeta _{j+1}]$. Thus, the event $\ED(\zeta _{j})$ can occur for at most  $\log ^3R $ different $j$ before $\tau \le \inner $. We obtain that
\begin{equation*}
\begin{split}
  \mathbb P \big( N_p \ge (\log ^4T)/p  \big)
  &\le \mathbb P \big( \zeta _{j_0}<\tau , \ |\{j< j_0: \ED(\zeta _j) \text{ holds} \}| \le \log ^3T \big)\\
    &\le \mathbb P \big( {\rm Bin}(j_0-1,p) \le \log ^3T \big) \le e^{-c\log ^3T},
\end{split}
\end{equation*}
where in the last inequality we used that $(j_0-1)p\ge c\log ^4T$ and a standard tail bound for the Binomial distribution (see, e.g., \cite[Corollary~A.1.14]{alon2016probabilistic}).
\end{proof}

\begin{lem}[Small Capacity]\label{lem:small-capacity}
    We have that $\mathbb P \big( {\rm Cap} _{\Lambda^+ }(A_\tau ) \ge r^2 \log ^6T  \big)\le e^{-c\log ^3T}$.
\end{lem}

\begin{proof}
For $A\subseteq [-4R,4R]^d$ and $z\in A\cap \Lambda ^+$, let $p(z,A):=\mathbb P (\Omega _{\rm esc}^{\Lambda }(z,A))$. Observe that if $z\in A'\subseteq A$ then $p(z,A)\le p(z,A')$ and therefore letting $S:=\{t\le \tau : \W(t)\in \Lambda ^+\setminus A_{t-1}\}$ be the set of times in which $\W$ discovers a new vertex $z\in \Lambda ^+$, we have that
\begin{align*}
  {\rm Cap} _{\Lambda ^+} (A_\tau )&=\!\!\sum _{z\in A_\tau \cap \Lambda ^+}\! p(z,A_\tau ) = \sum _{t\in S} p(\W(t),A_\tau ) \le \sum _{t\in S} p(\W(t),A_t )\\
  &\le C\sum _{t\in S}  \mathbb P \big( \ED(t)\mid \mathcal F _t \big),
\end{align*}
where in the last inequality we used \eqref{eq:escesc}. 
To bound the last sum, we first divide it into chunks of size $r^2$, namely, we define
\[
M_i\coloneqq \max \big\{ \mathbb P\big(\ED(t)\mid\mathcal F_t):t\in S\cap[ir^2,(i+1)r^2) \big\}
\]
or $M_i:=0$, if $S\cap [ir^2,(i+1)r^2)=\emptyset$. We get
\[
  \sum _{t\in S} \mathbb P \big( \ED(t)\mid \mathcal F _t \big)
  \le \sum _{i\le\lceil\tau/r^2\rceil} r^2M_i.
\]
We further break the sum according to the scale of $M_i$, namely let $P:=\{ 2^{-j} : 1\le j \le \lceil \log _2(r^{2d}) \rceil  \}$  be the set of relevant dyadic scales, and then
\[
\sum _{i\le\lceil\tau/r^2\rceil} r^2M_i
  \le 1+\sum_{p\in P} 2pr^2|\{i\le \lceil\tau/r^2\rceil:M_i\in[p,2p]\}|
  \le 1+\sum_{p\in P} 2pr^2N_p,
\]
where the summand $1$ comes from $i$ with $M_i\le r^{-2d}$ (note that the number of $i$ such that $M_i>0$ is at most $|S|\le |\Lambda ^+|\le Cr^d$). On the event $\mathcal B:=\bigcap _{p\in P}\{N_p\le (\log ^4T)/p\}$, the last sum is bounded by $2r^2(\log ^4T)|P| \le Cr^2 \log ^5 T$. This finishes the proof since $\mathbb P (\mathcal B^c)\le e^{-c\log ^3T}$ by Lemma~\ref{lem:Np}.
\end{proof}

\begin{proof}[Proof of Proposition~\ref{prop:heavy}]
    Since $\tau \le \tau _{\rm heavy}$, the set $A_\tau $ is nowhere heavy and therefore by Lemma~\ref{lem:Cap-Vol}, on the complement of the event in Lemma~\ref{lem:small-capacity} we have
    \begin{equation*}
        |A_\tau \cap \Lambda | \le (r^{6\nu }{\rm Cap}_{\Lambda^+ }(A_\tau )) ^{\frac{d}{d-2}} \le r^{10\nu } r^{\frac{2d}{d-2}} \le r^{\kappa }
    \end{equation*}
    and therefore $\tau _{\rm
    heavy }(\Lambda ) >\tau $. It follows from Lemma~\ref{lem:small-capacity} that $\mathbb P (\tau _{\rm
    heavy }(\Lambda ) \le \tau) \le e^{-c\log ^3T}$.
\end{proof}

\section{Demons}\label{sec:demons}

In this section we finish the proof of Theorem~\ref{thm:step} by proving Assumption~\hyperref[h2]{\bf{(H2)}} at time $T$. This is stated in the following theorem.

\begin{thm}\label{thm:relaxed}
We have that 
\begin{equation*}
     \mathbb P \bigg( \!\!\! \begin{array}{cc}
         \forall t\le \inner -R^{\epsilon } \text{ such that }\W(t)\in [-R,R]^d  \\
         \text{ we have } |(t,t+R^{\epsilon })\cap \, \DemonRel(R) | \ge 0.9R^{\epsilon }+1 
    \end{array} \bigg) \ge 1-Ce^{-c\log ^3T}.
\end{equation*}
\end{thm}
(Assumption \hyperref[h2]{(\bf{H2})} will follow for $a$ sufficiently small)

The rest of this section is devoted to the proof of Theorem~\ref{thm:relaxed}. To this end, fix $r:=\lceil R^{\delta } \rceil $ and  consider the following partition of space into (almost) disjoint blocks
\begin{equation*}
    \mathcal Q :=\big\{ \Lambda =u+[-r,r]^d : 
u\in [-\tfrac{3}{2}R,\tfrac{3}{2}R]^d \cap 2r\mathbb Z ^d  \big\}.
\end{equation*}
Similarly to \eqref{eq:spend}, for a block $\Lambda =u+[-r,r]^d\in \mathcal Q$ we define the stopping time $\spend (\Lambda )$ to be the first time $\W$ spends at least $r^2\log ^3r$ time in $\Lambda $:
\begin{equation*}
    \spend (\Lambda ):=\min \big\{ t\ge 0 : |\{s\le t: \W(s)\in \Lambda \}| \ge r^2\log ^3r\big\}
\end{equation*}
(so $\inner=\spend([-R,R]^d)$ but we will keep the compact notation $\inner$). We say that the block $\Lambda \in \mathcal Q$ is good if for any $t\le \spend (\Lambda )$ such that $\W (t)\in \Lambda $ we have that $|(t,t+r^{\epsilon })\cap \linebreak[0]{\DemonRel} (r)| \ge 0.9r^{\epsilon }+1$ (recall the definition of $\DemonRel$, Definition \ref{def:relaxed}). The block is said to be bad otherwise. Define the sets of blocks
\begin{equation}\label{eq:AB}
    \mathcal A:=\big\{ \Lambda \in \mathcal Q : \spend (\Lambda ) \le \inner \big\}  ,\quad \mathcal B:=\big\{ \Lambda \in \mathcal Q : \Lambda \text{ is bad} \big\}.
\end{equation}
Observe that if $\W (t)\in \Lambda $ for some $t\le \inner$ and $\Lambda \notin \mathcal A \cup \mathcal B$ then there is a high density of $r$-locally relaxed times in $[t,t+r^{\epsilon }]$. By Theorem~\ref{thm:heavy}, these times are actually $R$-locally relaxed as bigger blocks are not heavy. Hence, in order to prove Theorem~\ref{thm:relaxed} it suffices to bound the amount of time spent in blocks $\Lambda \in \mathcal A \cup \mathcal B$. In \S~\ref{sec:super} we bound the number of blocks in $\mathcal B$, in \S~\ref{sec:spend} we bound the number of blocks in $\mathcal A \setminus \mathcal B$, and in \S~\ref{sec:volume} we show that in any given block the walk cannot spend time much larger than its volume. In \S~\ref{subsec:relaxed} we collect all these ingredients and prove Theorem~\ref{thm:relaxed}.

\subsection{The demon supermartingale}\label{sec:super}

Our goal is to prove the kind of spatial independence required to boost the probabilities for the induction step. Concretely, we show that Assumption~\hyperref[h2]{\bf{(H2)}} with the demon gives an estimate on $k$ blocks that is a $k^\textrm{th}$ power of the estimate on one block given by the induction assumption.

There is a game-theoretic approach to this problem which we will not pursue. Imagine a backgammon grand master playing against multiple opponents. At each step the grand master can choose which board to play, and of course, what to play. She needs to beat all her opponents. It is intuitively obvious that, in fact, her ability to choose which opponent to play, and her ability to use information from other boards for the current board are useless. She just needs to play the optimal strategy in each board, independently. This can be proved formally by backward induction from terminal states of all boards back to the beginning of the game. We skip the details, which are easy. Our grandmaster is the walk in $[-4R,4R]^d$, the $k$ boards are $k$ different $r$-boxes and `winning a board' means making this box bad, i.e.~in~$\mathcal B$. 

Instead of this game-theoretic approach, we will use a more probabilistic point of view, using supermartingales. Let us first fix some notations. We have `two demons' involved --- one in scale $r$, which we get from the induction assumption, and one in scale $R$, which we need to control in order to prove that Assumption~\hyperref[h2]{\bf{(H2)}} holds at time $T$ (recall \S~\ref{sec:demon}). We denote them by $\demon^r$ and $\demon^R$, respectively.

Fix a block $\Lambda =u+[-r,r]^d$ from $\mathcal Q$ and let $\Lambda ^+:=u+[-4r,4r]^d$. Let $\demon^r$ be a demon with respect to $\Lambda^+$. Let $\Omega (\demon^r)$ be the event that $W_{\demon^r}$ satisfies the event from Assumption~\hyperref[h2]{\bf{(H2)}} for the block $\Lambda$, namely
\begin{equation*}
    \Omega (\demon^r):= \bigg\{ \!\! \begin{array}{cc}
        \forall t\le \spend (\Lambda )-r^{\epsilon } \text{ such that }W_{\demon^r} (t)\in \Lambda   \\
          \text{ we have }\big|(t,t+r^{\epsilon })\cap \, \textrm{Rel}_{\demon^r} (r) 
 \big| \ge  0.9r^{\epsilon }+1 
    \end{array}\!\!\! \bigg\}.
\end{equation*}
We wish to condition $\Omega (\demon^r)$ on $W_{\demon^r}[0,t]$ for some $t$, but we have to be a little careful, because a given teleporter $\gamma \in \mathscr{P}(\Lambda ^+) $ might be incompatible with the demon. Namely, every possible value of $\gamma$ has various values of $\demon^r$ for which
$
\PP(W_{\demon^r}[0,t]=\gamma)=0
$
and conditioning is impossible (simply because $\gamma$ might already have several entries into $\Lambda^+$ while $\demon^r$ would have entered from a different edge). Thus, whenever we condition on $W_{\demon^r}[0,t]$, we implicitly assume this is not the case.

With this convention, define
\begin{equation}\label{eq:sup}
    P(\Lambda ,\gamma ):= \sup _{\demon^r}\,\mathbb P \big( \Omega (\demon^r) ^c \,\big|\,W_{\demon^r}[0,\len \gamma]=\gamma\big),
\end{equation}
where, again, the supremum is taken only on values of $\demon^r$ compatible with $\gamma$.
In words, this is the probability that the event in the inductive hypothesis will not be satisfied inside $\Lambda $ given initial trajectory  $\gamma$ and given a worst case, adversarial future strategy.

We now return the demon $\demon^R$ outside the box $[-4R,4R]^d$. Consider
\[
M_\Lambda(t)=P\big( \restrict{\Lambda,W_{\demon^R}[0,t]}{\Lambda^+} \big).
\]
In other words, we take $W_{\demon^R}$ (up to time $t$), which is a random walk in $[-4R,4R]^d$ (with various places where it exits this box and is returned somewhere) and restrict it to $\Lambda^+$. The restriction gives  a teleporting path in $\Lambda^+$ with, again, various places where it exits $\Lambda^+$ and is returned somewhere. We feed this path into $P$, which calculates the probability that $\Omega (\demon^r)$ will fail with the worse \emph{local} demon which is compatible with the existing path.

\begin{lem}\label{lem:supermartingale}
    The process $M_\Lambda (t)$ is a supermartingale.
\end{lem}

\begin{proof} Denote by $\mathcal F_t$ the $\sigma$-field generated by $W_{\demon^R}(0),\dots, W_{\demon^R}(t)$. It is enough to show that $\EE[M_\Lambda(t+1) \mid \mathcal F_t]\le M_\Lambda(t)$.
  Suppose first that $W_{\demon^R}(t)\in \Lambda$. In this case, for every compatible $\demon^r$ we have that
  \[
  \PP\big(\Omega (\demon^r)^c\,\big|\,\restrict{W_{\demon^R}[0,t]}{\Lambda^+}\big)
  \]
  is a Doob martingale (a.k.a.\ an information exposure martingale). The set of $\demon^r$ compatible with the past does not change from $t$ to $t+1$, nor does the identity of the maximiser (the maximiser is the function that assigns to each path the best entry point, so it does not depend on which path is used). Hence, in this case $\mathbb E [M_{\Lambda }(t+1)\mid \mathcal F _t]=M_{\Lambda }(t)$.

  Next, suppose that $W_{\demon^R}(t)\notin \Lambda $. If $W_{\demon^R}(t+1)\notin \Lambda $ then $M_\Lambda (t+1)=M_\Lambda (t)$ and if $W_{\demon^R}(t+1)\in \Lambda $ then $M_\Lambda (t+1)\le M_\Lambda (t)$ since the walk did not necessarily enter from the worst possible edge (the supremum in \eqref{eq:sup} is taken over a smaller set). Hence, when $W_{\demon^R}(t)\notin\Lambda$ we have that $\mathbb E [M_{\Lambda }(t+1)\mid \mathcal F _t]\le M_{\Lambda }(t)$, completing the proof of the lemma.
\end{proof}

\begin{lem}\label{lem:product}
    Let  $\Lambda _1,\dots ,\Lambda _k\in \mathcal Q$ be $k$ blocks with $d(\Lambda _i^+,\Lambda _j^+)\ge 2$ for all $i\neq j$. Then
    \begin{equation*}
        \mathbb P \big( \forall i\le k, \  \Lambda _i\in \mathcal B \big) \le e^{-ck \log ^2 T}. 
    \end{equation*}
\end{lem}

\begin{proof}
  Define the process
\begin{equation*}
    M(t):=\prod _{j=1}^k M_{\Lambda _j} (t).
\end{equation*}
Of course, in general a product of supermartingales is not necessarily a supermartingale, but in this case it is. Indeed, we showed during the proof of Lemma~\ref{lem:supermartingale} that $M_{\Lambda _j}(t)$ is fixed whenever $d(W_{\demon^R}(t), \Lambda _j^+)>1$ and therefore, using that $\Lambda _j^+$ are separated, it follows that only one of the processes $M_{\Lambda _j }(t)$ changes in each time step, from which it follows that $M$ is a supermartingale. We obtain that 
\[
    \mathbb P \big( \forall i\le k, \ \Lambda _i\in \mathcal B \big) = \lim _{t\to \infty }  \mathbb E [M(t)] \le \mathbb E [M(0)] \le e^{-ck \log ^2 T},
    \]
    where the equality holds because we are in a finite box, so by infinite time we enter all the $\Lambda_i$ enough times to know if $\Omega (\demon^r)$ happened or not, and $M$ is either 0 or 1. In the last inequality we used that $M_{\Lambda _j}(0)\le e^{-c\log ^2 T}$ for all $j\le k$, which follows from our inductive Assumption~\hyperref[h2]{\bf{(H2)}} at time $4r^2$. 
\end{proof}

\begin{cor}\label{cor:bad}
    We have that $\mathbb P \big(  |\mathcal B| \ge \log T \big) \le e^{-c\log ^3T}$.
\end{cor}

\begin{proof}
If $|\mathcal B|\ge \log T$ then by the pigeonhole principle, we can find a subset $\mathcal B' \subseteq \mathcal B$ with $|\mathcal B'| = \lceil (\log T)/5^d \rceil $ such that any pair of blocks $\Lambda ,\Lambda '\in \mathcal B'$ satisfies $d(\Lambda ^+ ,(\Lambda ')^+)\ge 2$. Indeed, the blocks $\Lambda =u+[-r,r]^d \in \mathcal B \subseteq  \mathcal Q$ can be sorted into $5^d$ congruence classes according to \mbox{$u(\!\!\!\!\mod  10r \mathbb Z ^d)$}, with one of these classes containing at least $(\log T)/5^d$ of the blocks in $\mathcal B$. Clearly, within such a class, all blocks are separated as needed.

Thus, taking a union bound over all subsets $\mathcal B' \subseteq \mathcal Q$ of size $\lceil (\log T)/5^d \rceil$ such that $\min _{\Lambda ,\Lambda '\in \mathcal B'}d(\Lambda ^+ ,(\Lambda ')^+)\ge 2$, and using Lemma~\ref{lem:product} with $k:=\lceil (\log T)/5^d \rceil$ we obtain
\[
\mathbb P \big( |\mathcal B | \geq \log T \big)\le \sum _{\mathcal B '} \mathbb P\big( \text{every } \Lambda \in \mathcal B' \hbox{ is bad} \big) = \sum _{\mathcal B'} e^{-ck\log ^2T}  \leq  Ce^{-c\log ^3T},
\]
where in the last inequality we used that the number of subsets $\mathcal B'\subseteq \mathcal Q$ of size $k$ is bounded by
$|\mathcal Q|^k \leq e^{C\log ^2 T}$. The corollary follows (the $C$ can be removed since $T\ge a^{-1/(2d)}$ is sufficiently large).
\end{proof}

\subsection{The time spent in most blocks}\label{sec:spend}

In this section we bound the size of $\mathcal A\setminus \mathcal B$. Before stating and proving the lemma, let us explain why this is not immediate corollary of Theorem~\ref{thm:heavy} or, indeed, of any type of capacity argument. Theorem~\ref{thm:heavy} bounds the number of edges traversed in each $r$-block, but not the amount of time. Theoretically, it is possible that $\W$ spent a lot of time in a small box, going over and over the same small area. This would be difficult to preclude using a capacity argument like the one used in Theorem \ref{thm:heavy}, because the capacity does not increase when the walk goes over the same space over and over. So we need a different argument, which is supplied by Lemma \ref{lem:mathcal A}.

We no longer need the notation $\demon^r$, so whenever we write $\demon$ we mean the demon living outside $[-4R,4R]^d$ (what was denoted by $\demon^R$ in \S\ref{sec:super}).

\begin{lem}\label{lem:mathcal A}
    Let $k:=\lceil \log T \rceil $ and let $\Lambda _1,\dots ,\Lambda _k\in \mathcal Q$  such that  $d(\Lambda _i,\Lambda _j)\ge (\log ^2T)r$ for all $i\neq j$. Then, we have that 
    \begin{equation*}
        \mathbb P \big( \forall i \le k, \ \Lambda _i\in \mathcal A \setminus \mathcal B \big)\le e^{-c\log ^3T}.
    \end{equation*}
\end{lem}

\begin{proof}
The idea is that if $\Lambda _i$ is good then every time we enter it before $\spend (\Lambda _i)$ we will quickly find an $r$-locally relaxed time which is actually $R$-locally relaxed since bigger blocks are not heavy by Theorem~\ref{thm:heavy}. From that time we have a positive probability to escape $\Lambda _i$ in time $r^2$, avoid $\bigcup _{i\le k} \Lambda _i$, enter $[-R,R]^d$ and spend $R^2$ time in there. This event cannot happen more than $\log ^3R$ times before $\inner$. Hence, in order to have $\spend (\Lambda _i)\le \inner $ for all $i\le k$, the event above, that happens with positive probability, will have to fail too many times, which is very unlikely.

More precisely, define a sequence of stopping times by $\zeta _0=0$ and inductively for $j\ge 1$
    \begin{equation*}
        \zeta _{j}= \min \Big\{ t> \zeta _{j-1} +r^2 : \W(t)\in \smash{\bigcup _{i\le k}}\Lambda _i^+\text{ and } t\in \DemonRel (R)  \Big\}.
    \end{equation*}
First, we prove that $\W$ has a positive chance to spend $R^2$ time inside the inner box $[-R,R]^d$ in the interval $[\zeta _j,\zeta _{j+1}]$. Define the event
\begin{equation*}
    \Omega _j := \Big\{ \W[\zeta _j+r^2, \zeta _j +2R^2 ] \cap   \smash{\bigcup _{i\le k}}\Lambda _i^+ =\emptyset \Big\} \cap \big\{ \W[\zeta _j+ R^2, \zeta _j +2R^2 ]\subseteq [-R,R]^d  \big\}.  
\end{equation*}
We claim that $\mathbb P (\Omega _j \mid \mathcal F _{\zeta _j })\ge c$, where $\mathcal F_t :=\sigma (\W(0),\dots \W(t))$. To this end, let $W'$ be a once-reinforced walk starting from $\W(\zeta _j)$ and independent of $\mathcal F _{\zeta _j
}$ except for its starting location (formally, we shift an independent once-reinforced walk starting from $0$ by $\W(\zeta _j)$). Define the event
\begin{equation*}
    \Omega _j ':= \Big\{ W'[r^2, 2R^2 ] \subseteq (\W(\zeta _j) +[-R,R]^d) \setminus  \smash{\bigcup_{i\le k}}\rule[-1.8ex]{0pt}{0pt}\Lambda _i^+  \Big\} \cap \big\{  W'[R^2, 2R^2 ] \subseteq [-R,R]^d  \big\}.  
\end{equation*}
It follows from the Brownian approximation that $\mathbb P (\Omega _j'\mid \mathcal F _{\zeta _j})\ge c$. Indeed, by Lemma~\ref{lem:3.3}, there is a coupling of $W'$ with a Brownian motion $B$ with diffusive constant $\sigma ^2 \in [\frac{0.9}{d},\frac{1.1}{d}]$ starting from $\W(\zeta _j)$ such that the event $\mathcal G :=\big\{\!\max _{t\le 2R^2}\|W'(t)-B(t)\|\le (t+r^2)^{1/3+5\epsilon  } \big\}$ holds with probability at least $1-e^{-c\log ^2T}$. Define the event
\begin{equation}\label{eq:Brownian2}
    \Omega _j '':= \Big\{  B[r^2, 2R^2 ]^+ \subseteq  (\W(\zeta _j) +[-R,R]^d) \setminus  \smash{\bigcup_{i\le k}}\Lambda _i^+  \Big\} \cap \big\{  B[R^2, 2R^2 ]^+ \subseteq [-R,R]^d  \big\},
\end{equation}
where $B[a,b]^+$ is the set $\big\{ y\in \mathbb R^d : \exists t\in [a,b] , \ \|y-B(t)\|\le t^{2/5} \big\}$, i.e.\ an inflating Wiener sausage. A straightforward Brownian motion estimate shows that $\mathbb P (\Omega _j''\mid \mathcal F _{\zeta _j} ) \ge c$. For completeness, we prove this in the appendix in Lemma \ref{lem:Brownian2} (in this lemma $x=\W(\zeta _j)$ and we drop the conditioning using that $B$ is independent of $\mathcal F _{\zeta _j}$ except for its starting position). We have that $\Omega _j'' \cap \mathcal G \subseteq \Omega _j ' $ and therefore $\mathbb P (\Omega _j'\mid \mathcal F _{\zeta _j} ) \ge c$.

In order to bound the probability of $\Omega _j$ we need to show that $W'$ couples with $\W$. This is done as in Lemma~\ref{lem:avoiding}. Define the set $A:=\W [0,\zeta _j] \cap (\W(\zeta _j)+[-R,R]^d) $ and the event $\mathcal E := \big\{ \forall s\in [a^{-2/3},2R^2], \ W'(s)\notin A \big\}$. By Lemma~\ref{lem:A} with $L=a^{-1/3}$ we have that $\mathbb P (\mathcal E)\ge 1-Ca^{1/9}$ (here it was important that we defined $\zeta_j$ using times relaxed at scale $R$). Let $\mathcal H := \{\forall s\le \min (a^{-2/3}, 2R^2), \  W'(s)=\W (\zeta _j +s) \}$ and observe that $\mathbb P (\mathcal H)\ge 1-Ca^{1/3}$. On the event $\Omega _j'\cap \mathcal E \cap \mathcal H$, the walk $\W[\zeta _j,\zeta _j+2R^2]$ will not exit $\W(\zeta _j) +[-R,R]^d$ and will not interact with its history there, and therefore $W'(s)=\W(\zeta _j +s)$ for all $s\le 2R^2$. We obtain that $\Omega _j'\cap \mathcal E \cap \mathcal H\subseteq \Omega _j$. This completes the proof that $\mathbb P (\Omega _j \mid \mathcal F _{\zeta  _j})\ge c$.

Next, we let $j_0:=\lfloor \log ^4r \rfloor $ and claim that 
\begin{equation}\label{eq:s}
\begin{split}
    \big\{ \forall i\le k, \ \Lambda _i &\in  \mathcal A \setminus \mathcal B \big\} \cap \{\tau _{\rm heavy } > \inner \} \\
    &\subseteq   \big\{ \big|\{j\le j_0: \zeta _j \le \inner \text{ and }\Omega _j \text{ fails}\}\big| \ge j_0-\log ^3R \big\}
\end{split}
\end{equation}
To see this, suppose that the left hand side of \eqref{eq:s} holds and define the set of times 
\begin{equation}
 S:=\big\{t\le \inner -r^{\epsilon }: \exists i\le k \text{ with }\W(t) \in \Lambda _i \text{ and } t\le \tau _{\rm spend}(\Lambda _i)-r^\epsilon  \big\}. 
\end{equation}
By the definition of $\mathcal A$ \eqref{eq:AB} we have that $|S|\ge k(r^2\log ^3r-r^{\epsilon })\ge 3r^2\log ^4r$. Moreover, by the definition of $\mathcal B$, for any $t\in S$ we have that $[t,t+r^\epsilon ] \cap \DemonRel(r)\neq \emptyset $ and since $t+r^{\epsilon }\le \inner <  \tau _{\rm heavy}$ we actually have $[t,t+r^\epsilon ] \cap \DemonRel(R)\neq \emptyset $. Hence, for any $j\ge 1$, if $t\in S\cap [\zeta _j+r^2,\zeta _{j+1}]$ then $\zeta _{j+1}\in [t,t+r^{\epsilon }]$ and therefore $|S\cap [\zeta _j,\zeta _{j+1}]| \le r^2+r^{\epsilon } \le  2r^2$. Thus, at least $j_0+1$ of the stopping times $\zeta _j$ must occur before $\inner$. However, the event $\Omega _j$ implies that $\W$ spends at least $R^2$ time in $[-R,R]^d$ in the time interval $[\zeta _j,\zeta _{j+1}]$ and therefore $\Omega _j$ cannot occur more than $\log ^3R$ times before $\inner$. This completes the proof of \eqref{eq:s}. Using \eqref{eq:s} we obtain 
\begin{equation*}
\begin{split}
    \mathbb P \big(& \forall i \le k, \ \Lambda _i\in \mathcal A \setminus \mathcal B \big) \\
    &\le \mathbb P (\inner \ge \tau _{\rm heavy}) +\mathbb P \big( \big| \{j\le j_0 : \zeta _j\le \inner \text{ and } \Omega _j \text{ fails}\}\big| \ge j_0 -\log ^3R \big) \\
    &\le e^{-c\log ^3T} + \mathbb P \big( {\rm Bin}(j_0,1-c) \ge j_0-\log ^3R \big) \le 2e^{-c\log ^3T},
\end{split}
\end{equation*}
where in the second inequality we used Theorem~\ref{thm:heavy} and in the last inequality we used a standard tail bound for the Binomial distribution (see, e.g., \cite[Corollary~A.1.14]{alon2016probabilistic}).
\end{proof}

\begin{cor}\label{cor:spend}
    We have that $\mathbb P \big( |\mathcal A \setminus \mathcal B| \ge (\log T)^{3d} \big)\le e^{-c\log ^3T}$.
\end{cor}

\begin{proof}
    The proof is identical to Corollary~\ref{cor:bad} but with a slightly larger separation between the blocks required in Lemma~\ref{lem:mathcal A}. This is done by sorting the blocks $\Lambda =u+[-r,r]^d \in \mathcal A\setminus \mathcal B$ into congruence classes $u(\!\!\!\!\mod  2\lceil \log ^2T \rceil r \mathbb Z ^d)$ instead of $u(\!\!\!\!\mod  10r \mathbb Z ^d)$. This gives $c(\log T)^d$ separated blocks which is more than $k=\lceil \log T \rceil $ from Lemma~\ref{lem:mathcal A}, as long as $T\ge a^{-1/(2d)}$ is sufficiently large. 
\end{proof}

\subsection{Volume exhaustion}\label{sec:volume}

\begin{lem}[Volume exhaustion]\label{lem:Volume exhaustion}
    Before $\inner$ we visit every vertex in $[-2R,2R]^d$ at most $r^{2d}$ times with probability at least $1-e^{-c\log ^3T}$. Namely, we have that 
    \begin{equation*}
        \mathbb P \big( \exists x\in [-2R,2R]^d : |\{s\le \inner : \W(s)=x \}|\ge r^{2d}  \big)\le e^{-c\log ^3T}.
    \end{equation*}
\end{lem}

For the proof of the lemma we will need the following claim about a simple random walk. As in Lemma \ref{lem:3.4}, we will use $D(0,L)$ to denote a Euclidean ball of radius $L$.

\begin{claim}[Exiting through a specific edge]\label{claim:simple}
    Let $X$ be a simple random walk starting from the origin and let $L>0$. For any edge $e$ connecting $D(0,L)$ to $\mathbb Z^d\setminus D(0,L)$, the probability that $X$ exits $D(0,L)$ for the first time through $e$ before returning to $0$ is at least $cL^{1-d}$.
\end{claim}

\begin{proof}
    Consider the random walk Green's function defined by 
    \begin{equation*}
        G(x):=\sum _{t=1}^{\infty} \mathbb P (X(t)=x).
    \end{equation*}
    It is well known (see, e.g., \cite[Theorem~4.3.1]{lawler2010random}) that  $G(x)=c_1\|x\|^{2-d}+O(\|x\|^{-d})$, as $x$ tends to infinity for some explicit, dimension dependent constant $c_1$.
    
    We may assume that $L$ is sufficiently large. Let $y\in \mathbb Z ^d$ such that $\|y\|_2\in [L-2,L-1]$ and let $Y$ be a simple random walk starting from $y$. The function $G$ is harmonic everywhere except for zero and therefore the process $G(Y(t\wedge \tau ))$ is a martingale where 
    \begin{equation*}
        \tau :=\min \big\{t\ge 0 : Y(t)=0 \text{ or }Y(t)\notin D(0,L) \big\}.
    \end{equation*}
    Since $\tau $ is finite almost surely and $G$ is bounded it follows that 
    \begin{equation*}
        c_1\|y\|^{2-d}-C\|y\|^{-d} \le G(y)=\mathbb E [G(Y(\tau ))]\le G(0)\cdot \mathbb P (Y(\tau )=0)+ c_1L^{2-d}+CL^{-d} .
    \end{equation*} 
    Using that $\|y\|^{2-d}-L^{2-d}\ge cL^{1-d}$ we obtain $\mathbb P (Y(\tau )=0)\ge cL^{1-d}$ for large $L$. Inverting time, this means that the random walk $X$ starting from $0$, has probability at least $cL^{1-d}$ to reach $y$ before returning to zero or exiting $D(0,L)$. Choosing $y$ at a constant distance from the edge $e$ finishes the proof of the claim (since once $X$ reaches $y$ it has a constant probability to take the right steps and exit through $e$). 
\end{proof}

\begin{proof}[Proof of Lemma~\ref{lem:Volume exhaustion}]
    The idea is that if $\W$ visits a vertex $x$ more than $r^{2d}$ times, then it is very likely that it completely covered the block $\Lambda =x+[-r,r]^d$ but we already know that there are no heavy blocks before $\inner$ with high probability. More precisely, let $\zeta _i$ be the $i^{\rm th}$ return to $x$. Let $y_i$ be the closest point in Euclidean distance from $x$ such that there is an edge $e_i$ containing $y_i$ that was never crossed before time $\zeta _i$. Let $\eta _i$ be the first time after $\zeta _i$ in which $\W$ crosses $e_i$. We claim that on the event $\zeta _i<\tau _{\rm heavy}$ we have that
    \begin{equation}\label{eq:88}
        \mathbb P \big( \eta _i \le \zeta _{i+1} \mid \mathcal F _{\zeta _i} \big) \ge cr^{1-d}.
    \end{equation}
    To see this, observe that after $\zeta _i$, the walk $\W$ is a simple random walk until it exits the Euclidean ball $D(x,r_i)$ of radius $r_i=\|x-y_i\|_2$, since all the edges in $D(x
    ,r_i)$ are reinforced. Moreover, if $\zeta _i<\tau _{\rm heavy}$ then $r_i<r$ since otherwise the block $\Lambda =x+[-r,r]^d$ is heavy. Thus, by Claim~\ref{claim:simple} the walk $\W$ has probability at least $cr_i^{1-d}\ge cr^{1-d}$ to exit the ball $D(x,r_i)$ through the edge $e_i$ before returning to $x$, finishing the proof of \eqref{eq:88}. 
    
    The number of new edges $e_i$ in $\Lambda $ that can be crossed before $\tau _{\rm heavy}(\Lambda )\le \tau _{\rm heavy}$ is at most $Cr^{\kappa }<r^d$. Thus, letting $k=r^{2d}$ we have by \eqref{eq:88} 
    \begin{equation*}
    \begin{split}
        \mathbb P ( \zeta _k \le  \tau _{\rm heavy} ) \le \mathbb P \big( \big| \big\{ i\le k : \zeta _i<\tau _{\rm heavy} \text{ and } \eta _i\le \zeta _{i+1} \big\}\big| \le r^d \big)\\
        \le \mathbb P \big( {\rm Bin}(k,cr^{1-d}) \le r^d \big) \le e^{-cr},
    \end{split}
    \end{equation*}
    where in the last inequality we used, e.g., \cite[Corollary~A.1.14]{alon2016probabilistic}. By Theorem~\ref{thm:heavy} this shows that $\mathbb P (\zeta _k\le \inner)\le e^{-c\log ^3T}$, finishing the proof of the lemma.
\end{proof}

\subsection{Relaxed times}\label{subsec:relaxed}
By now we have all the ingredients required for the proof of  Theorem~\ref{thm:relaxed}.

\begin{proof}[Proof of Theorem~\ref{thm:relaxed}]
Define the set of times 
\begin{equation*}
    S:= \big\{ s\le \inner -r^{\epsilon } : \W (s)\in \Lambda \text{ for some }\Lambda \in \mathcal A \cup \mathcal B \big\}
\end{equation*}
and note that by the definition of $\mathcal A$ and $\mathcal B$ \eqref{eq:AB}, for any $s\notin S$ with $s\le \inner -r^{\epsilon }$ we have
\begin{equation}\label{eq:2}
 \big| (s,s+r^{\epsilon })\cap \DemonRel (r) \big|\ge 0.9r^{\epsilon }+1.   
\end{equation}
By Lemma~\ref{lem:Volume exhaustion}, Corollary~\ref{cor:spend}, and Corollary~\ref{cor:bad} we have that 
\begin{equation}\label{eq:S}
    \mathbb P \big( |S| \ge r^{3d} \big) \le e^{-c\log ^3T}.
\end{equation}
Next, we fix $t\le \inner -R^{\epsilon }$ and show that there are many relaxed times in $(t,t+R^{\epsilon })$. Define a sequence of times by $s_0:=\min \{s\ge t:s\notin S\}$ and inductively for $i\ge 1$
\begin{equation*}
    s_i:=\min \{s\ge s_{i-1}+r^{\epsilon }-1 : s\notin S\}.
\end{equation*}
Let $k$ be the last $i$ for which $s_i\le t+R^{\epsilon }-r^{\epsilon }$. We let $I_i:=(s_i,s_i+r^{\epsilon })\cap \mathbb Z$ (so that the $I_i$ are disjoint) and observe that by \eqref{eq:2} we have $|I_i\setminus \DemonRel(r)|\le 0.1r^{\epsilon }-1$. If in addition $\tau _{\rm heavy}>\inner $ then in fact these times are relaxed at scale $R$, not just $r$, i.e.~we have $|I_i\setminus \DemonRel(R)|\le 0.1r^{\epsilon }-1$. 

Define $I\coloneqq(t,t+R^{\epsilon })\cap \mathbb Z$. Observe that any $s\in I$ is inside one of the intervals $I_i$, unless $s\in S$ or $s>t+R^{\epsilon }-r^{\epsilon }$. Thus, on the event that $\tau _{\rm heavy}>\inner $ and $|S|<r^{3d }$ we have 
\begin{equation*}
    \big| I \setminus \DemonRel (R) \big| \le |S|+r^{\epsilon } + \sum _{i=1}^{k}    \big| I_i \setminus \DemonRel (R) \big| \le r^{3d }+r^\epsilon + k(0.1r^{\epsilon }-1)  \le 0.1 R^{\epsilon } -2,  
\end{equation*}
where in the last inequality we used that $kr^{\epsilon }<R^{\epsilon }$ since the $I_i$ are disjoint (so $k\le\lceil R^\eps-1\rceil/\lceil r^\eps-1\rceil$) while on the other hand $k\approx R^{\epsilon }/r^{\epsilon } \gg r^{3d }$ since $r=R^\delta$ and $\delta \le \epsilon /(10d)$\label{pg:del_eps}. This completes the proof of Theorem \ref{thm:relaxed} using Theorem~\ref{thm:heavy} and \eqref{eq:S}. As explained on page \pageref{sec:demons}, Theorem \ref{thm:step} also follows.
\end{proof}

\section{Proof of Theorem~\ref{thm:1}}\label{pg:proof1}

    Using the induction base in Lemma~\ref{lem:base} and the induction step in Theorem~\ref{thm:step}, we get that \hyperref[h1]{\bf{(H1)}} and~\hyperref[h2]{\bf{(H2)}} hold at any time $T\ge T_0$. It follows immediately from Assumption~\hyperref[h1]{\bf{(H1)}} on transition probabilities that $W$ is transient. Next, we prove that $W$ scales to a Brownian motion.
    For any $T\ge T_0$ let $\sigma _T:=\sqrt{\var (W(T)_1)/T}$ where  $W(T)_1$ is the first coordinate of $W(T)$. Once the induction has been established for all $T\ge T_0$, it follows from the further clause of Proposition~\ref{prop:var} that $\sigma _T$ is a Cauchy sequence (be careful that in \S~\ref{sec:coupling} the standard deviation $\sigma _T$ is simply denoted by $\sigma $). The main clause of Proposition~\ref{prop:var} implies that the limit, $\sigma $, satisfies that $\sigma \in [\frac{0.9}{\sqrt{d}},\frac{1.1}{\sqrt{d}}]$. We now apply Lemma~\ref{lem:coupling}. Since Lemma~\ref{lem:coupling} holds for any $T\ge a^{-1/(2d)} $ and $t\le T$, we can replace $\sigma _T$ from that lemma with the limiting $\sigma $. This gives that for all $t\ge 1$, there is a coupling of $W$ with a standard Brownian motion $B$ in $\mathbb R ^d$ such that 
    \begin{equation*}
        \mathbb P \Big( \max _{s\le t} \|W(s)-\sigma B(s)\| \ge t^{1/3+3\epsilon } \Big) \le e^{-c\log ^2t}.
    \end{equation*}
    From this estimate the convergence clearly follows.\qed

\appendix
\section{Brownian estimates}

The next standard lemma follows from the fact that $\|x\|_2^{2-d}$ is harmonic, so $\|B(t)\|_2^{2-d}$ is a martingale. For a complete proof see, e.g., \cite[Theorem~3.18]{morters2010brownian}.

\begin{lem}\label{lem:Rr}
    Let $d\ge 3$ and let $B$ be Brownian motion in $\mathbb R ^d$ starting from some $x\in \mathbb R ^d$. Let $r\in (0,\|x\|_2)$ and $R\in (\|x\|_2,\infty ]$. Let $\tau :=\inf \big\{ t>0 :\|B(t)\|_2\in \{r,R\} \big\}$. Then, 
    \begin{equation}
        \mathbb P \big(  \|B(\tau )\|_2 =r  \big) = \frac{ \|x\|_2^{2-d}-R^{2-d}}{r^{2-d}-R^{2-d} }.
    \end{equation}
\end{lem}

The following standard lemma is taken from \cite[Lemma~9.2]{elboim2022infinite}.

\begin{lem}\label{l:BM.ball.exit}
Let $\{B(t)\}_{t\ge 0}$ be Brownian motion in $\mathbb R ^d$ with diffusive constant $\sigma \in [1/M,M]$ starting from some $x\in \mathbb  R^d$. Then, for all $s\ge r>0$ we have
\[
\mathbb P\Big( \inf_{t\geq s^2} \|B(t)\|_2 \leq r \Big) \leq C_M (r/s)^{d-2}.
\]
\end{lem}

To state the following lemma we will need some notations. For $x,y\in \mathbb R ^d$ and $r>0$ we define the cylinder ${\rm cyl}(x,y,r)$ by 
\begin{equation}
    {\rm cyl}(x,y,r):=\big\{ z\in \mathbb R ^d: \exists w\in [x,y] \text{ with } \|z-w\|_2\le r \big\},
\end{equation}
where $[x,y]\subseteq \mathbb R ^d$ is the line segment connecting $x$ and $y$. Next, for a Brownian motion $B$ in $\mathbb R ^d$ with $B(0)=x$, we will use the notation $B[0,t]\rightsquigarrow {\rm cyl}(x,y,r)$ to denote the event that $B$ travels through the cylinder ${\rm cyl}(x,y,r)$ in the time interval $[0,t]$. Namely,
\begin{equation*}
    \big\{B[0,t]\rightsquigarrow {\rm cyl}(x,y,r)\big\} := \big\{\forall s\in [0,t],\ B(s)\in  {\rm cyl}(x,y,r) \big\} \cap \big\{ \|B(t)-y\|_2\le r\big\}.
\end{equation*}

\begin{lem}[Traveling in a cylinder]\label{lem:tube}
    For any $M>0$ there exists $c_M>0$ such that the following holds. Let $t>0$ and let $x,y\in \mathbb R ^d$ such that $\|x-y\|_2\le M\sqrt{t}$. Let $B$ be a Brownian motion starting from $x$ with diffusive constant $\sigma \in [1/M,M]$. Then,
    \begin{equation}
        \mathbb P \big( B[0,t] \rightsquigarrow {\rm cyl}(x,y,\sqrt{t}/M)  \big)\ge c_M.
    \end{equation}
\end{lem}

\begin{proof}
    We may assume that $x=0$ and by Brownian scaling that $t=1$. By increasing $M$ and scaling we may further assume that $B$ is a standard Brownian motion. By rotational invariance we may assume that $y=(y_1,\dots ,y_d)$ satisfies $y_1>0$ and $y_i=0$ for all $i\in [2,d]$. By increasing $M$ we may replace the norm $\|\cdot \|_2$ with $\|\cdot \|_{\infty }$. After all of these reductions our event becomes a product event which means that we may assume that $d=1$. The remaining case follows from, e.g., \cite[Exercise~1.8]{morters2010brownian} with the linear function $f(t)=ty$.
\end{proof}

\begin{lem}[Escaping into a ball]\label{lem:brownian}
    For any $M\ge 1$ there exists $c_M>0$ such that the following holds. Let $R>0$ and $\ell \in (0,\frac9{10}R]$. Let $x\in \mathbb R ^d$ with $\|x\|<R$ and let $B$ be Brownian motion starting from $x$ with diffusive constant $\sigma \in [1/M,M]$. Then, for all $t\in [\ell ^2/M,M\ell ^2]$ we have that 
    \begin{equation}
        \mathbb P \big( \forall s\le t, \ \|B(s)\|_2\le R\text{ and } \|B(t)\|_2\le R-\ell \big) \ge c_M \cdot  \min \big ((R-\|x\|_2 ) / \ell , 1 \big). 
    \end{equation}
\end{lem}

\begin{proof}
    By Lemma~\ref{lem:Rr}, letting $\tau :=\inf \big\{ t>0 :\|B(t)\|_2\in \{R-\ell ,R\} \big\}$, we have that $\mathbb P (\|B(\tau )\|_2=R-\ell ) \ge c\cdot  \min \big( (R-\|x\|_2)/\ell ,1 \big)$, where we note that this probability is one when $\|x\|\le R-\ell $. Once we hit $\tau $ and $\|B(\tau )\|=R-\ell$, we have a positive chance to stay in the ball $\{y:\|y\|_2<R\}$ during the time interval $[\tau ,\tau +t]$ (this follows from, e.g., Lemma~\ref{lem:tube} with $x=y=B(\tau )$). Thus, the event 
    \begin{equation}
        \Omega _1:= \big\{ \|B(\tau )\|_2=R-\ell \text{ and } \forall s\in [\tau ,\tau +t], \ \|B(s)\|_2 \le R  \big\}
    \end{equation}
    holds with probability at least $c\cdot \min \big( (R-\|x\|_2)/\ell ,1 \big)$. Next, let $\Omega _2:= \{\|B(t)\|_2\le R-\ell \}$ and observe that $\mathbb P (\Omega _2)\ge c$, since $B(t)$ is a normal variable with variance of order $ \ell ^2$. Recall that $\{\|B(s)\|_2\}_{s\ge 0}$ is a Bessel process and note that both $\Omega _1$ and $\Omega _2$ are decreasing in $\{\|B(s)\|_2\}_{s\ge 0}$. Thus, by the FKG inequality for Bessel processes \cite[Corollary~1.10]{legrand2024some} we have that $\mathbb P (\Omega _1\cap \Omega _2)\ge c\cdot \min \big( (R-\|x\|_2)/\ell ,1 \big)$. This finishes the proof of the lemma since on $\Omega _1\cap \Omega _2$ the event of the lemma holds. 
\end{proof}

\begin{lem}\label{lem:Om2p}
If $R$ is sufficiently large then $\Omega_2'(x)$ from \eqref{eq:brownian1} satisfies $\mathbb P(\Omega_2'(x))>c$.
\end{lem}

Recall that $r\in [R^{\delta },R]$, $u\in [-2R,2R]^d$, $\Lambda =u+[-r,r]^d$ and $\Lambda ^+=u+[-2r,2r]^d$. Recall also that $x\in \Lambda $, that $B$ is Brownian motion starting at $x$ with diffusive constant $\sigma$ and that
\begin{equation*}
        \Omega _2'(x) :=  \bigg\{  \!\! \begin{array}{cc}
        B [0,\frac{1}{2}r^2]\subseteq u+[-\frac{5}{4}r,\frac{5}{4}r]^d, \ \forall t\in [30r^2,3R^2], \  d(B(t),\Lambda ^+)> t^{2/5} \\ B[\frac{1}{2}r^2,3R^2]\subseteq [-\frac{7}{3}R,\frac{7}{3}R]^d, \  B[\frac{3}{4}R^2,3R^2]\subseteq [-\frac{10}{12}R,\frac{10}{12}R]^d 
    \end{array} \!\! \bigg\}.
    \end{equation*}

\begin{proof}
Suppose first that $r\le R/M$ for a sufficiently large constant $M$. Fix a point $y\in [-2R,2R]^d$ with  $\|x-y\|_2=Mr$ and a point $z\in [-\frac{1}{2}R,\frac{1}{2}R]^d$ with  $\|x-z\|_2\ge R/4$. Define the events 
\begin{equation*}
\begin{split}
    &\Sigma _1:= \{B[0,\tfrac{1}{2}r^2]\rightsquigarrow {\rm cyl}(x,u,r/4)\},\quad \Sigma _2:=\{B[\tfrac{1}{2}r^2,20r^2]\rightsquigarrow {\rm cyl}(B(\tfrac{1}{2}r^2),y,r)\}\\
    &\Sigma _3:=\{B[20r^2, \tfrac{1}{2}R^2]\rightsquigarrow {\rm cyl}(B(20r^2),z,\tfrac{1}{8}R)\},\quad \Sigma _4:=\{B[\tfrac{1}{2}R^2,3R^2]\rightsquigarrow {\rm cyl}(B(\tfrac{1}{2}R^2),z,\tfrac{1}{8}R)\}\\
    &\Sigma _5:=\{\forall t\ge 20 r^2, \ \|B(t)-x\|\ge \sqrt{M}r \},\quad \Sigma _6:=\{\forall t\ge 20r^2,\ \|B(t)-x\|\ge 2t^{2/5}\}.
\end{split}
\end{equation*}
On $\Sigma _5\cap \Sigma _6$, for all $t\ge 20r^2$ we have 
\[
d(B(t),\Lambda^+)\ge \|B(t)-x\| -Cr\stackrel{\smash{\mathclap{\Sigma_5, \Sigma _6}}}{\ge} \max \big( \sqrt{M}r,2t^{2/5} \big) -Cr >t^{2/5}  .
\]
Thus, it is easy to check that $\Sigma _1\cap \Sigma _2\cap \Sigma _3\cap \Sigma _4\cap \Sigma _5\cap \Sigma _6\subseteq \Omega _2'(x)$ for $M$ sufficiently large. We turn to bound the probability of this intersection. By Lemma~\ref{lem:tube} we have that $\mathbb P (\Sigma _1\cap \Sigma _2)\ge c_M$ and $\mathbb P (\Sigma _3\cap \Sigma _4\mid \Sigma _1\cap \Sigma _2)\ge c$ where $c$ is independent of $M$. By Lemma~\ref{lem:Rr} (with $R_{\textrm{Lemma \ref{lem:Rr}}}=\infty $ and $r_{\textrm{Lemma \ref{lem:Rr}}}=\sqrt{M}r$) we have that $\mathbb P (\Sigma _5 \mid \Sigma _1\cap \Sigma _2) \ge 1-CM^{1-d/2}$. Finally, letting $k_0:=\lfloor \log _2(20r^2) \rfloor $ we have by Lemma~\ref{l:BM.ball.exit}
\begin{equation*}
    \mathbb P (\Sigma _6^c \mid \Sigma _1\cap \Sigma _2) \le \sum _{k=k_0}^{\infty } \mathbb P \Big( \min _{s\ge 2^{k-1}}\|B(s)-x\|_2 \le 2^{2k/5+1} \Big) \le \sum _{k=k_0}^{\infty } Ce^{-ck}  \le Cr^{-c}.
\end{equation*}
This shows that $\mathbb P (\Sigma _3\cap \Sigma _4\cap \Sigma _5\cap \Sigma _6 \mid \Sigma _1\cap \Sigma _2)\ge c$ as long as $M$ is sufficiently large. It follows that $\mathbb P (\Omega _2'(x))\ge c_M$.

Next, suppose that $r\in [R/M,R/3]$. In this case, we can fix $y\in [-\frac{9}{12}R,\frac{9}{12}R]^d$ such that $d(y,\Lambda ^+)\ge R/24$ and define the events
\begin{equation*}
    \Sigma _2':=\{B[\tfrac{1}{2}r^2,\tfrac{3}{4}r^2]\rightsquigarrow {\rm cyl}(B(\tfrac{1}{2}r^2),y,\tfrac{1}{100}r)\},\quad \Sigma _3':=\{B[\tfrac{3}{4}r^2, 3R^2]\rightsquigarrow {\rm cyl}(B(\tfrac{3}{4}r^2),y,\tfrac{1}{100}R)\}.
\end{equation*}
On $\Sigma _2'\cap \Sigma _3'$ we have that $d(B(t),\Lambda ^+)\ge R/50\ge t^{2/5}$ for all $t\in [30r^2,3R^2]$ for $R$ sufficiently large and therefore $\Sigma _1\cap \Sigma _2'\cap \Sigma _3'\subseteq \Omega _2'(x)$. This shows that $\mathbb P (\Omega _2'(x))\ge c_M$ using Lemma~\ref{lem:tube}.

Finally, when $r\in [R/3,R]$ we do not necessarily have a $y\in [-\frac{9}{12}R,\frac{9}{12}R]^d$ with $d(y,\Lambda ^+)\ge R/24$. However, in this case we have $30r^2>3R^2$ and therefore the second condition in the definition of $\Omega _2'(x)$ is empty. Hence, in this case we can define the events $\Sigma _2'$ and $\Sigma _3'$ as above with $y=0$ and still have $\Sigma _1\cap \Sigma _2'\cap \Sigma _3'\subseteq \Omega _2'(x)$, completing the proof.
\end{proof}

\begin{lem}\label{lem:Brownian2}
If $R$ is sufficiently large then $\mathbb P(\Omega_j'')>c$ where $\Omega_j''$ is from \eqref{eq:Brownian2}.
\end{lem}

Recall that
\begin{equation*}
    \Omega _j '':= \Big\{  B[r^2, 2R^2 ]^+ \subseteq  (x+[-R,R]^d) \setminus  \smash{\bigcup_{i\le k}}\Lambda _i^+  \Big\} \cap \big\{  B[R^2, 2R^2 ]^+ \subseteq [-R,R]^d  \big\}.
\end{equation*}
Here $r=\lceil R^{\delta }\rceil $, $k=\lceil \log T \rceil $, and $\Lambda_i^+=u_i+[-4r,4r]^d$ where $u_i\in[-\frac32 R,\frac32 R]^d$ and the $u_i$ are separated by $(\log ^2T)r$. $B$ is Brownian motion starting from some $x$ with $\|x-u_i\|\le Cr$ for some $i\le k$. Finally, recall that $B[a,b]^+$ is the inflating Wiener sausage and is given by $\big\{ y\in \mathbb R^d : \exists t\in [a,b] , \ \|y-B(t)\|\le t^{2/5} \big\}$.  

\begin{proof}
Let $u_1,\dots ,u_k$ be the centers of $\Lambda _1,\dots ,\Lambda _k$ respectively. Let $M$ be a sufficiently large constant and let $y\in [-2R,2R]^d$ be a point at distance $Mr$ from $x$. Since $u_i\in [-\frac{3}{2}R,\frac{3}{2}R]^d$, there exists some $z\in [-\frac{7}{8}R,\frac{7}{8}R]^d\cap (x+[-\frac{7}{8}R,\frac{7}{8}R]^d)$. Define the events
\begin{equation*}
\begin{split}
    &\Sigma _1:=\big\{ B[0,r^2]\rightsquigarrow {\rm cyl}(x,y,r) \big\},\quad \Sigma _2:= \big\{ B[r^2,R^2]\rightsquigarrow {\rm cyl}(B(r^2),z,R/20)  \big\}\\
    &\Sigma _3:= \big\{ B[R^2,2R^2]\rightsquigarrow {\rm cyl}(B(R^2),z,R/20)  \big\},\quad \Sigma _4^j:=\big\{ \forall  t\ge r^2,  \|B(t)-u_j \|\ge \sqrt{M}r \big\}\\
    &\Sigma _5^j:= \big\{ \forall  t\ge r^2,  \|B(t)-u_j \|\ge 2t^{2/5} \big\}
\end{split}
\end{equation*}
Clearly, for large $M$ we have that $\Sigma _1\cap \Sigma _2\cap \Sigma _3\cap \bigcap _{j\le k} (\Sigma _4^j\cap \Sigma _5^j)\subseteq \Omega _j''$.
By Lemma~\ref{lem:tube} we have that $\mathbb P (\Sigma _1)\ge c_M$ and $\mathbb P (\Sigma _2\cap \Sigma _3 \mid \Sigma _1)\ge c$. By the same arguments as the above proof we have that $\mathbb P (\Sigma _4^i\mid \Sigma _1)\ge 1-CM^{1-d/2} $ and similarly by Lemma~\ref{lem:Rr}, for $j\neq i$ we have  $\mathbb P (\Sigma _4^j\mid \Sigma _1)\ge 1-C_M(\log T)^{4-2d}$. Moreover, for all $j\le k$ we have that $\mathbb P (\Sigma _5^j \mid \Sigma _1) \ge 1-Cr^{-c}$. This shows that $\mathbb P (\Sigma _2\cap \Sigma _3\cap \bigcap _{j\le k} (\Sigma _4^j\cap \Sigma _5^j) \mid \Sigma _1) \ge c$ for large $M$, finishing the proof.
\end{proof}

\bibliography{Interchange}
\bibliographystyle{abbrv}

\end{document}